\documentclass[onefignum,onetabnum]{siamart220329}

\usepackage{lipsum}
\usepackage{amsfonts,amssymb}
\usepackage{graphicx}
\usepackage{epstopdf}
\usepackage{algorithmic}
\usepackage{color}
\usepackage{ulem}
\usepackage{comment}
\usepackage{bm}
\usepackage{tikz}
\usepackage{mathtools}
\usepackage{subcaption}
\usepackage{multirow}
\usepackage{threeparttable}

\usetikzlibrary{intersections,calc,arrows.meta}

\usepackage{booktabs} % For better table formatting

\newcommand*{\T}{\mathsf{T}}

\newcommand*{\paren}[1]{\left(#1\right)}
\newcommand*{\parsq}[1]{\left[#1\right]}
\newcommand*{\parbr}[1]{\left\{#1\right\}}

\DeclareMathOperator{\tr}{tr}
\DeclareMathOperator{\diag}{diag}

\DeclareMathOperator{\vecspan}{span}
\renewcommand{\figurename}{Figure }
\renewcommand{\tablename}{Table }
\newcommand*{\averes}{\bar{\mathcal{R}}}
\newcommand*{\gdir}{\mathcal{G}_{\mathrm{dir}}}
\newcommand*{\gundir}{\mathcal{G}_{\mathrm{undir}}}
\newcommand*{\din}{\delta_{\mathrm{in}}}
\newcommand*{\dout}{\delta_{\mathrm{out}}}
\newcommand*{\tbcell}[1]{\begin{tabular}{l}#1\end{tabular}}
\newcommand*{\pe}{P_{\varepsilon}}

\newcommand*{\correction}[1]{\textcolor{black}{#1}}
\newcommand*{\correctionx}[1]{\textcolor{black}{#1}}

\ifpdf
  \DeclareGraphicsExtensions{.eps,.pdf,.png,.jpg}
\else
  \DeclareGraphicsExtensions{.eps}
\fi

\crefname{ALC@unique}{Step}{Steps}

\newsiamremark{remark}{Remark}
\newsiamthm{problem}{Problem}
\crefname{problem}{Problem}{Problems}
\newsiamthm{assumption}{Assumption}
\crefname{assumption}{Assumption}{Assumptions}

\headers{Consensus Algorithm on SC Graph}{T.~YONAIYAMA AND K.~SATO}

\title{Performance bound analysis of linear consensus algorithm on strongly connected graphs using effective resistance and reversiblization\thanks{Submitted to the editors DATE.
\funding{This work was supported by Japan Society for the Promotion of Science KAKENHI under 23K28369.}
}}

\author{Takumi Yonaiyama\thanks{Department of Mathematical Informatics, Graduate School of Information Science and Technology, The University of Tokyo (T.~Yonaiyama: \email{yonaiyama-takumi844@g.ecc.u-tokyo.ac.jp}, K.~Sato: \email{kazuhiro@mist.i.u-tokyo.ac.jp}).}
\and Kazuhiro Sato\footnotemark[2]}

\usepackage{amsopn}
\begin{document}

\maketitle
% REQUIRED
\begin{abstract}
We study the performance of the linear consensus algorithm on strongly connected directed graphs
using the linear quadratic (LQ) cost as a performance measure.
In particular, we derive bounds on the LQ cost by leveraging effective resistance and reversiblization. Our results extend previous analyses---which were limited to reversible cases---to the nonreversible setting. To facilitate this generalization, we introduce novel concepts, termed the back-and-forth path and the pivot node, which serve as effective alternatives to traditional techniques that require reversibility. Moreover, we apply our approach to Cayley graphs and random geometric graphs to estimate the LQ cost without the reversibility assumption. The proposed approach provides a framework that can be adapted to other contexts where reversibility is typically assumed.
\end{abstract}

% REQUIRED
\begin{keywords}
  Directed graph, Effective resistance, Linear consensus, Reversiblization
\end{keywords}

\section{Introduction}
\label{sec: introduction}
The analysis of multiagent networks has been applied to a variety of subjects, such as multi-robot systems \cite{robot}, wireless sensor networks \cite{WSN}, and large language models \cite{LLM}. One of the simplest and most fundamental algorithms in this context is linear consensus, also known as average consensus. The linear consensus algorithm is used when each agent has a scalar value and has to update its value toward the weighted average of initial values, relying only on limited and decentralized communication. Despite its simplicity, the linear consensus algorithm forms the foundation for many tasks in network coordination, such as formation control \cite{formation1,multivehicle2}, distributed optimization \cite{optim1}, and cooperative leader following \cite{leader0,leader1}.

The performance measurement of the linear consensus algorithm has been studied in \correction{various ways}. A linear consensus network can be understood as \correction{a linear} discrete-time dynamical system closely related to a Markov chain \cite{cayley}, making Markov chain theory a key analytical tool. In fact, the performance of the linear consensus algorithm has been analyzed by the speed of convergence using spectral analysis of the transition matrix of the associated Markov chain \cite{spectral1,spectral2}. On the other hand, recent studies often adopt approaches from a \correction{control-theoretic} point of view, such as linear quadratic (LQ) cost, which is classical in control. The LQ cost is defined as the \correction{sum of squared} $L_2$-norm of the difference between the state trajectory and the final state. Moreover, the LQ cost appears in the error estimation of noisy consensus, a variant of linear consensus with additive random noise. 
Motivated by these connections, several studies have employed the LQ cost or similar functions as performance measures \cite{cayley,aboutlqcost,mainp}. In these studies, various approaches are employed:
\begin{itemize}
  \item \textbf{Eigenvalue calculation}: \correction{Eigenvalues} are calculated for special graph types, including Cayley graphs, grid graphs, and random geometric graphs \cite{cayley,aboutlqcost}.
  \item \textbf{Effective resistance}: The concept, which bridges electrical networks with Markov chains \cite{randomwalks}, has been employed in several recent studies \cite{citeres1,satokazu}. Notably, previous research has estimated the LQ cost using the effective resistance \cite{mainp}, enabling the assessment of LQ cost based on the network's graph topology.
  \item \textbf{Hitting time of a Markov chain} (see \cite{revmarkov}): Recent research \cite{novelnoisy} has provided an exact formulation of a variant of the LQ cost by leveraging hitting times.
\end{itemize}

In this paper, we derive both upper and lower bounds on the LQ cost for nonreversible cases using effective resistance. This generalizes the estimation of the LQ cost for reversible linear consensus networks, as studied in \cite{mainp}, to nonreversible networks. A key advantage of nonreversible Markov chains is its applicability to directed networks. While the reversibility is a useful assumption, it does not hold in directed communication scenarios,
making the nonreversible setting highly relevant.
Thus, recent studies have highlighted the importance of
 nonreversible Markov chains \cite{nonrev1,nonrev2,nonrev3}.  
 \correction{A central technique in this line is} called \textit{reversiblization} of the transition matrix, which allows one to exploit results of studies on reversible Markov chains \cite{reversiblization}.

Our main contribution is the establishment of an upper bound on the LQ cost for nonreversible cases, and a lower bound under an assumption weaker than reversibility using the effective resistance of reversiblization of the underlying Markov chain. To obtain these bounds on the LQ cost, we use the following methods:
\begin{enumerate}
  \item Although the analogy between the effective resistance and Markov chain does not hold for nonreversible cases, we demonstrate that through the reversiblization of a nonreversible chain \correction{(see \Cref{subsec:reversiblization})}, the effective resistance can be applied to a nonreversible linear consensus system.
  \item We relate the reversiblization and the original nonreversible Markov chain by introducing new notions named the \textit{back-and-forth path} and \textit{pivot node}. These concepts serve as alternatives to the notion of $2$-\textit{fuzz} of an undirected graph \cite{hfuzz}, which is central in reversible cases, but cannot directly apply to nonreversible cases. \correction{These concepts, mainly used in \Cref{bounds_GPP}, demonstrate how graph-theoretic methods can be adapted to the analysis of nonreversible cases.}
\end{enumerate}
The differences between our approach to upper and lower bounds and those of previous studies are summarized in Table \ref{table:appr}. 
In all methods except our own, reversibility is typically assumed in the analysis of linear consensus, owing to the availability of numerous results pertaining to reversible Markov chains. 
 
\correction{
Furthermore, to clarify the behavior of our bounds, we introduce some specific applications and experiments, including Cayley graphs and random geometric graphs. Cayley graphs are highly symmetric and thus analytically tractable, while random geometric graphs model real-world communication networks. These graphs have been extensively studied in the literature (see \cite{rggs} and \cite{cayley}).}

\begin{table}[h]
  \caption{Approaches to the LQ cost}
  \label{table:appr}
  \centering
  \tabcolsep = 3pt
    \begin{tabular}{c|lll}
    \hline
    Method & Target class & Main tool & Key technique \\
    \hline\hline
    \tbcell{\cite{cayley}} & \tbcell{Cayley, grid,\\and random\\geometric graphs} & \tbcell{Calculation of\\eigenvalues} & \tbcell{Defining an appropriate\\trigonometric polynomial}\\
    \hline
    \tbcell{\cite{mainp}} & \tbcell{Reversible} & \tbcell{Effective resistance} & \tbcell{$2$-fuzz of a graph}\\
    \hline
    \tbcell{\cite{novelnoisy}} & \tbcell{Reversible} & \tbcell{Hitting time of\\a Markov chain} & \tbcell{Markov's inequality}\\
    \hline
    \tbcell{Proposed\\method} & \tbcell{\textbf{Nonreversible}} & \tbcell{Effective resistance} & \tbcell{``Back-and-forth path''\\and ``Pivot node''\\of a graph}\\
    \hline
    \end{tabular}
\end{table}

The remainder of this paper is organized as follows: In Section 2, we introduce the mathematical formulation of linear consensus algorithm, effective resistance, and the relationship between them. In Section 3, we  present our main result on the estimation of the LQ cost, along with its proofs. Section 4 demonstrates applications \correction{and numerical experiments}, and Section 5 concludes the paper.

\subsection*{Notations}
Let $\mathbb{R}$ be \correction{the} set of real numbers, $\mathbb{R}^n$ be \correction{the} set of $n$-dimensional real vectors, and $\mathbb{R}^{n\times m}$ be \correction{the} set of $n\times m$ real matrices. We denote \correction{the} $i$-th entry of $\bm{v}\in\mathbb{R}^n$ by $v_i$, and \correction{the} $(i,j)$ entry of $A\in\mathbb{R}^{n\times m}$ by $A_{ij}$. Let $\bm{e}_i\in\mathbb{R}^n$ be \correction{the} vector whose $i$-th entry is $1$ and other entries are $0$, and $\bm{1}\in\mathbb{R}^n$ be \correction{the} vector with all entries equal to $1$. We denote \correction{the} identity matrix by $I$ and \correctionx{the} zero matrix by $O$. For a vector $\bm{v}\in\mathbb{R}^n$, let $\bm{v}^\T$ be \correction{the} transpose of $\bm{v}$, $\vecspan\{\bm{v}\}$ be a vector space spanned by $\bm{v}$, and $\diag(\bm{v})\in\mathbb{R}^{n\times n}$ be \correction{the} diagonal matrix satisfying $\diag(\bm{v})_{ii}=v_i$ for all $i=1,2,\dots,n$ and $\diag(\bm{v})_{ij}=0$ if $i\neq j$. For a matrix $A$, let $A^\T$ be \correction{the} transpose of $A$, $\tr A$ be \correction{the} trace of $A$, $\ker A$ be \correction{the} kernel of $A$, and $A^{\dagger}$ be \correction{the} Moore-Penrose pseudoinverse of $A$. For two matrices $A$ and $B$, $A\geq B$ means $A_{ij}\geq B_{ij}$ for all $i$ and $j$. We denote a graph with a set of \correction{nodes} $V$ and a set of \correction{edges} $\mathcal{E}$ by $(V,\mathcal{E})$. For an undirected graph, we denote an edge between nodes $u$ and $v$ by $\{u,v\}$, while for a directed graph, we denote an edge from node $u$ to node $v$ by $(u,v)$.

We also introduce the following notation for central objects in this paper. The formal definitions are given in Section~\ref{sec:preliminaries}.

\begin{table}[h]
  \caption{Central objects in this paper}
  \label{table:objects}
  \centering
  \tabcolsep = 3pt
    \begin{tabular}{cl}
    \tbcell{$P$} & \tbcell{A consensus matrix}\\
    \hline
    \tbcell{\color{black}$\bm{\pi}$} & \tbcell{\color{black}The invariant measure of $P$}\\
    \hline
    \tbcell{\color{black}$\varPi$} & \tbcell{\color{black}The diagonal matrix defined as $\diag(\bm{\pi})$}\\
    \hline
    \tbcell{$P^*$} & \tbcell{\color{black}The matrix defined as $\varPi^{-1} P^\T\varPi$}\\
    \hline
    \tbcell{$C$} & \tbcell{A conductance matrix}\\
    \hline
    \tbcell{$J(P)$} & \tbcell{The linear quadratic cost of the consensus matrix $P$}\\
    \hline
    \tbcell{$J_{\mathrm{w}}(P)$} & \tbcell{The weighted version of the linear quadratic cost of\\the consensus matrix $P$}\\
    \hline
    \tbcell{$\gdir(P)$} & \tbcell{The associated \textbf{directed} graph of $P$ (the graph where\\the edge $(u,v)$ exists if and only if $P_{uv}>0$)}\\
    \hline
    \tbcell{$\mathcal{G}(P)$} & \tbcell{The associated \textbf{undirected} graph of $P$ (the graph where\\the edge $\{u,v\}$ exists if and only if $P_{uv}>0$ or $P_{vu}>0$}\\
    \hline
    \tbcell{$\mathcal{R}_{uv}(C)$} & \tbcell{The effective resistance between the node $u,v$\\in the network with the conductance matrix $C$}\\
    \hline
    \tbcell{$\mathcal{R}_{uv}(\mathcal{G}(P))$} & \tbcell{The effective resistance between the node $u,v$ in the graph $\mathcal{G}(P)$\\(setting the conductance of all edges to $1$)}\\
    \hline
    \tbcell{$\averes(C),\averes(\mathcal{G}(P))$} & \tbcell{The average effective resistance of the network\\with the conductance matrix $C$ or the graph $\mathcal{G}(P)$}\\
    \end{tabular}
\end{table}

\color{black}

%%%%%%%%%%%%%%%%%%%%%%%%%%%%%%%%%%%%%%%%%%%%%%%%%%%%%%%%%%%%%%%%%%%%%%%%%%%
\section{Preliminaries}
\label{sec:preliminaries}
\subsection{Definition of linear consensus algorithm}
We model a communication network of $n$ agents labeled $1$ to $n$. Each pair of agents can communicate according to the directed graph $\gdir=(V,\mathcal{E})$, where $V=\{1,2,\dots,n\}$ represents the set of \correction{agent labels}, and $(i,j)\in\mathcal{E}$ means that $i$ can get information from $j$. We call $\gdir$ a communication graph. In this paper, the information is represented by a real number.

The linear consensus algorithm ensures that all agents converge to the weighted average of their initial real numbers by exchanging information according to a predefined graph. In this algorithm, at each iteration, each agent transmits its current number to all connected agents and subsequently updates its own number to a convex combination of the received numbers using predetermined coefficients. More precisely, the linear consensus algorithm repeats the following update:
\[\bm{x}(t+1)=P\bm{x}(t),\]
indicating that
\begin{eqnarray}\label{fmlxt}
  \bm{x}(t)=P^t\bm{x}(0),
\end{eqnarray}
where $P\in\mathbb{R}^{n\times n}$ is a row stochastic matrix, namely, satisfying $P\bm{1}=\bm{1}$, and $\bm{x}(t)\in\mathbb{R}^n$ represents the number held by each agent at time $t$. In this paper, we refer to row stochastic matrices simply by \textit{stochastic matrices}.

The constraint that agents can only communicate according to the communication graph \correction{$\gdir$ is reflected in the placement} of nonzero elements in $P$, \correction{that is, the edge $(u,v)$ appears in $\gdir$ if and only if $P_{uv}\neq 0$. Furthermore, we define the \textit{undirected graph associated with matrix} $P$, $\mathcal{G}_{\mathrm{undir}}(P)=(V,\mathcal{E})$, letting $\{u,v\}\in\mathcal{E}$ if and only if $P_{uv}\neq 0$ or $P_{vu}\neq 0$. In this paper, we denote $\gundir=\mathcal{G}$ for simplicity.}

Throughout this paper, we assume two properties: $P$ is \textit{irreducible}, and the diagonal elements of $P$ are positive, where $P$ is \textit{irreducible} if and only if $\gdir(P)$ is strongly connected. The positive diagonal assumption may appear to be strong, but it is natural for each agent to use its current state for each iteration, and easy to implement since it does not require any additional communication. In addition, this assumption leads to some useful properties, such as \textit{aperiodicity} of $P$, which means that the greatest common divisor of the lengths of all cycles in $\gdir(P)$ is $1$. Aperiodicity is a desirable property and is often enforced by imposing additional assumptions. For example, the PageRank algorithm calculates the stationary distribution of a Google matrix, which is arranged from the initial hyperlink matrix to make it aperiodic \cite{pagerank}.

Under these two assumptions, the Perron-Frobenius theorem \cite{perronfrobenius} implies that $P$ has the eigenvalue $1$ with the multiplicity $1$, the corresponding right eigenvector is $\bm{1}$, and corresponding left eigenvector is a strictly positive vector. We denote the \correction{left} eigenvector by \correction{$\bm{\pi}^\T$} normalized so that $\sum_u \pi_u=1$, and call \correction{$\bm{\pi}$} the \textit{invariant measure} of $P$. Moreover, as $t \to \infty$, $P^t$ converges to $\bm{1}\bm{\pi}^\T$. Because of (\ref{fmlxt}), we obtain $\lim_{t\to\infty}\bm{x}(t)=\bm{1}\bm{\pi}^\T\bm{x}(0)$. We also use \correction{the} diagonal matrix $\varPi:=\diag(\bm{\pi})$.

In this paper, we call a stochastic and irreducible matrix whose diagonal elements are strictly positive, a \textit{consensus matrix}.

\subsection{Performance measure of linear consensus algorithm}
\label{subsec:perf_lc}
As a performance measure, we use the linear quadratic (LQ) cost \cite{mainp}
\begin{eqnarray}\label{lqcost}
  J(P):=\frac{1}{n}\sum_{t\geq 0}\|P^t-\bm{1}\bm{\pi}^\T\|_{\mathrm{F}}^2=\frac{1}{n}\tr\left[\sum_{t\geq 0}(I-\bm{\pi}\bm{1}^\T)(P^\T)^t P^t(I-\bm{1}\bm{\pi}^\T)\right],
\end{eqnarray}
where $\|\cdot\|_\mathrm{F}$ means the Frobenius norm of a matrix.

The cost (\ref{lqcost}) is obtained by evaluating $\mathrm{E}\left[\sum_{t\geq 0}\|\bm{x}(t)-\bm{x}(\infty)\|^2\right]$, where $\mathrm{E}[\cdot]$ means the expected value, under the assumption that $\bm{x}(0)$ is a random vector with covariance $\mathrm{E}\left[\bm{x}(0)\bm{x}(0)^\T\right]=I$ \cite{aboutlqcost,mainp}. In addition, the cost (\ref{lqcost}) also appears in the noisy consensus \cite{aboutlqcost,mainp,novelnoisy}. The noisy consensus repeats the update
\[\bm{x}(t+1)=P\bm{x}(t)+\bm{n}(t),\]
where $\bm{n}(t)$ is an independent and identically distributed process with the average $\mathrm{E}\left[\bm{n}(t)\right]=\bm{0}$ and the covariance $\mathrm{E}\left[\bm{n}(t)\bm{n}(t)^\T\right]=I$. Assume that $\bm{x}(0)$ is a random vector with covariance $\mathrm{E}\left[\bm{x}(0)\bm{x}(0)^\T\right]=I$ and is not correlated with $\bm{n}(t)$. To measure the distance between $\bm{x}(t)$ and its weighted average $\bm{1}\bm{\pi}^\T\bm{x}(t)$, we define $\tilde{\bm{e}}(t):=(I-\bm{1}\bm{\pi}^\T)\bm{x}(t)$. Then, we can show that
\begin{eqnarray}\label{noisy1}
  \frac{1}{n}\lim_{t\to\infty}\mathrm{E}\parsq{\|\tilde{\bm{e}}\correctionx{(t)}\|^2}&=&J(P).
\end{eqnarray}

From the context of noisy consensus, we also use another variant \cite{novelnoisy}
\begin{eqnarray}\label{lqcost2}
  J_{\mathrm{w}}(P):=\tr\left[\sum_{t\geq 0}(I-\bm{\pi}\bm{1}^\T)(P^\T)^t\varPi P^t(I-\bm{1}\bm{\pi}^\T)\right],
\end{eqnarray}
which is obtained by substituting $\|\tilde{\bm{e}}\|^2$ in (\ref{noisy1}) \correctionx{with} $\sum_i \pi_i\tilde{e}_i^2$. \correction{Unlike} (\ref{lqcost}), which sums up the errors uniformly, (\ref{lqcost2}) sums up the errors with weights according to $\bm{\pi}$.

\subsection{Effective resistance}
We consider a resistor network as an undirected connected graph $\mathcal{G}=(V,\mathcal{E})$ \correction{where} each edge represents the resistor that connects both nodes of the edge. A resistor network with $n$ nodes is determined by assigning a matrix $C=(C_{ab})\in\mathbb{R}^{n\times n}$ whose element $C_{ab}$ \correction{represents the conductance} of the edge between $a$ and $b$ if $C_{ab}\neq 0$, or that $a$ and $b$ are not connected by an edge if $C_{ab}=0$. We call $C=(C_{ab})$ a \textit{conductance matrix} if $C$ is a symmetric, nonnegative and irreducible matrix.

We consider a voltage vector of the nodes $\bm{v}$, where $v_a$ represents the voltage of the node $a$, and a current vector $\bm{i}$, where $i_a$ represents the current flowing out of (or into when negative) the node $a$. By Ohm's law, the current from $a$ to $b$ through an edge $\{a,b\}$ is $C_{ab}(v_a-v_b)$. Therefore, by Kirchhoff's law, $i_a$ is determined by
\begin{eqnarray}
  i_a=\sum_{b\in V}C_{ab}(v_a-v_b)\quad
{\rm or}\quad
\bm{i}=L(C) \bm{v}, \label{kirch1}
\end{eqnarray}
where $L(C):=\diag(C\bm{1})-C$ is the Laplacian of $C$. Notice that the diagonal elements of $C$ are not concerned with $L(C)$.

Because the current injection and extraction are balanced, we can assume $\bm{i}^{\T}\bm{1}=0$. Under this condition, \correction{the following lemma is well known fact (e.g., Lemma 6.12 in \cite{FB-LNS}).}
\begin{lemma}\label{v-uniqueness}
  If $\bm{i}^{\T}\bm{1}=0$, then (\ref{kirch1}) has a unique solution for $\bm{v}$ up to the fundamental solution term $\bm{1}$. That is, for each current vector $\bm{i}$, the potential difference between nodes $a$ and $b$ is unique.
\end{lemma}

Consider the situation \correction{where} a voltage source is connected between nodes $a,b$, and adjust the voltage so that a unit current flows through the source. In this setting, the \correction{current} vector $\bm{i}$ is $\bm{e}_a-\bm{e}_b$, and $v_a-v_b$, the potential difference between nodes $a,b$, is unique. Thus, we can define \textit{effective resistance} between $a,b$:
\begin{definition}[\cite{mainp}]
  Let $\mathcal{G}=(V,\mathcal{E})$ be an undirected graph with the conductance matrix $C$. The \textit{effective resistance} between nodes $a$ and $b$, denoted by $\mathcal{R}_{ab}(C)$, is defined by $v_a-v_b=(\bm{e}_a-\bm{e}_b)^\T\bm{v}$, where $\bm{v}$ is an arbitrary solution of (\ref{kirch1}) when $\bm{i}=\bm{e}_a-\bm{e}_b$.
\end{definition}

The effective resistance $\mathcal{R}_{ab}(C)$ can also be expressed as
\begin{eqnarray}
  \mathcal{R}_{ab}(C)=(\bm{e}_a-\bm{e}_b)^\T L(C)^{\dagger}(\bm{e}_a-\bm{e}_b),
\end{eqnarray}
because $\bm{v}=L(C)^\dagger\bm{i}$ \cite{kronred}.

In our paper, we also consider resistor networks whose resistors have \correction{unit conductance}. For such networks, we only have to indicate the unweighted graph $\mathcal{G}$, so we denote the effective resistance between $a,b$ for such networks by $\mathcal{R}_{ab}(\mathcal{G})$, which is the property determined only by the graph topology.

In addition, we define \[\averes(C):=\frac{1}{2n^2}\sum_{u,v\in V}\mathcal{R}_{uv}(C)\] as the average effective resistance and $\averes(\mathcal{G})$ is defined using $\mathcal{R}_{uv}(\mathcal{G})$ in the same way. $\averes(\mathcal{G})$ is also determined only by the graph topology.

\subsection{The relationships between linear consensus algorithm and effective resistance}
In this section, we summarize some relationships between \textit{reversible} consensus matrices and conductance matrices.

A consensus matrix $P$ can be treated as the transition matrix of a discrete time Markov chain with the stationary distribution $\bm{\pi}$. A Markov chain with the transition matrix $P$ is called \textit{reversible} if and only if $\pi_xP_{xy}=\pi_{y}P_{yx}$ holds for all $x,y$, that is, the matrix $\varPi P$ is symmetric. Analogously, we call a consensus matrix $P$ \textit{reversible} if and only if $\varPi P$ is symmetric. For reversible $P$, we can associate $P$ with a conductance matrix $C$:
\begin{lemma}
  Let $S_{\mathrm{cs}}\subseteq \mathbb{R}^{n\times n}$ be the set of reversible consensus matrices, and for $\alpha>0$, let \[S_{\alpha}=\{C\in\mathbb{R}^{n\times n}\mid C^\T =C,\ C\geq 0,\ C_{ii}> 0\ (\forall i),\ \bm{1}^\T C\bm{1}=\alpha,\ C\textrm{ is irreducible}\}\]be the set of conductance matrices whose sum of all elements is $\alpha$. Then, $\Phi_{\alpha}(P):=\alpha\varPi P$ is a bijection between $S_{\mathrm{cs}}$ and $S_{\alpha}$, and the inverse is $\Psi(C):=\diag(C\bm{1})^{-1}C$.
\end{lemma}
The proof is analogous to \correction{the argument of Section 3.1.2 in} \cite{mainp}.

Let $\Phi(P):=\Phi_n(P)=n\varPi P$, where $n$ denotes the number of rows of $P$. Then, we can see a clear relationship between \correction{the random walk with the transition matrix $P$ and some properties in the electrical network with conductance matrix $\Phi(P)$ such as the effective resistance, the voltage and the current} (see \cite{randomwalks}). In the analogy of random walk, a matrix called Green matrix \cite{mainp} plays an important role. The Green matrix of consensus matrix $P$ is defined as
\begin{eqnarray}\label{defgp}
  G(P):=\sum_{t\geq 0}(P^t-\bm{1}\bm{\pi}^\T).
\end{eqnarray}
This matrix is also called the fundamental matrix \cite{revmarkov}. For a consensus matrix $P$, this matrix is well defined (see Chapter \correction{11.5} in \cite{greenconv}).
\begin{remark}
  Because $P^t-\bm{1}\correction{\bm{\pi}^\T}=(P-\bm{1}\bm{\pi})^t$ for $t\geq 1$, \eqref{defgp} can be reformulated as
  \begin{eqnarray}\label{defgpref}
    G(P)+\bm{1}\correction{\bm{\pi}^\T}=I+\sum_{t\geq 1}(P-\bm{1}\bm{\pi}^\T)^t.
  \end{eqnarray}
  Moreover, since $(I-A)^{-1}=\sum_{t\geq 0}A^t$ if the sum is well defined, we obtain \begin{eqnarray}\label{defgpfinal}
    G(P)+\bm{1}\correction{\bm{\pi}^\T}=(I-P+\bm{1}\bm{\pi}^\T)^{-1}.
  \end{eqnarray}
  A fundamental matrix often refers not to $G(P)$, but to the \correction{right-hand side} of \eqref{defgpref} or \eqref{defgpfinal} (e.g. \cite{greenconv,novelnoisy}).
\end{remark}

The reason for introducing the Green matrix is that for reversible consensus matrix $P$, the effective resistance of $C:=\Phi(P)$ can be written as
\[\mathcal{R}_{ab}(C)=\frac{1}{n}(\bm{e}_a-\bm{e}_b)^\T G(P)\varPi^{-1}(\bm{e}_a-\bm{e}_b),\]
which can be found in \cite{mainp}.
\color{black}
\correction{Such properties} in this section \correction{are} used to relate $J(P)$ and the effective resistance of the graph. In particular, $G(P)$ plays an important role in evaluating $J(P)$.

%%%%%%%%%%%%%%%%%%%%%%%%%%%%%%%%%%%%%%%%%%%%%%%%%%%%%%%%%%%%%%%%%%%%%%%%%
\section{Performance analysis of linear consensus algorithm using effective resistance}

In this section, we give upper and lower bounds for $J(P)$ using the concept of effective resistance. A bound for $J(P)$ has been known for reversible $P$ \cite{mainp,novelnoisy}, but not for nonreversible cases. To provide bounds for nonreversible $P$, we use a reversible matrix $P^*P$, where $P^*:=\varPi^{-1}P^\T\varPi$. Our result is a generalization of \cite{mainp}.

\subsection{Properties of $P^*P$}\label{subsec:reversiblization}
We show some properties of $P^*P$. For every consensus matrix $P$, $P^*P$ is a reversible consensus matrix with the same invariant measure as $P$, so it is called \correction{\textit{a multiple reversiblization}} of $P$ \cite{reversiblization}. When $P$ is a reversible consensus matrix, we have $P^*=\varPi^{-1}(\varPi P)^\T=\varPi^{-1}(\varPi P)=P$ and $P^*P=P^2$.

The following lemma plays an important role in this paper.
\begin{lemma}\label{tracekey}
  For every consensus matrix $P$ and nonnegative integer $t$, 
  \begin{eqnarray}\label{tracekeymain}
    \tr\paren{(P^*)^tP^t}\leq\tr \paren{(P^*P)^t}.
  \end{eqnarray}
\end{lemma}
\begin{proof}
  Let $Q:=\varPi^{1/2}P\varPi^{-1/2}$. Then, $\tr\paren{(P^*)^tP^t}=\tr\paren{(Q^t)^\T Q^t}=\sum_{i}\sigma_i^2(Q^t)$, where $\sigma_i(Q)$ is the $i$-th singular value of $Q$. Moreover, $\tr \paren{(P^*P)^t}=\tr \paren{(Q^\T Q)^t}=\sum_{i}\sigma_i^{2t}(Q)$, because $(Q^\T Q)^t$ is symmetric. Using $\sigma_i(Q^t)\leq \sigma_i^t(Q)$, as shown in 9.H.2.a of \cite{traceineq}, we obtain \eqref{tracekeymain}.
\end{proof}
% \begin{remark}
%   We have asssumed that the diagonal elements of $P$ are positive. Without this assumption, even if $P$ is apperiodic and irreducible, $P^*P$ can not be irreducible matrix. Let
%   \begin{eqnarray*}
%     P:=\begin{pmatrix}
%       0 & 1 & 0 \\
%       0 & 0 & 1 \\
%       \frac{1}{2} & 0 & \frac{1}{2}
%     \end{pmatrix},
%   \end{eqnarray*}
%   then $\bm{\pi}=\begin{pmatrix}\frac{1}{4}&\frac{1}{4}&\frac{1}{2}\end{pmatrix}^\T$ and
%   \begin{eqnarray*}
%     P^*:=\begin{pmatrix}
%       0 & 0 & 1 \\
%       1 & 0 & 0 \\
%       0 & \frac{1}{2} & \frac{1}{2}
%     \end{pmatrix}.
%   \end{eqnarray*}
%   Therefore,
%   \begin{eqnarray*}
%     P^*P:=\begin{pmatrix}
%       \frac{1}{2} & 0 & \frac{1}{2} \\
%       0 & 1 & 0 \\
%       \frac{1}{4} & 0 & \frac{3}{4}
%     \end{pmatrix},
%   \end{eqnarray*}
%   then $P^*P$ is not irreducible. This matrix has an eigenvalue $1$ with multiplicity $2$.

%   Similarly, if some of the diagonal elements of $P$ is very small, 
% \end{remark}

The equality of \eqref{tracekeymain} holds not only for reversible matrices. For example, normal consensus matrices satisfy the property $P^*P=PP^*$. The class of normal consensus matrices includes Cayley graphs, and $J(P)$ on Cayley graphs have been studied (e.g. \cite{cayley}).
\begin{proposition} \label{Prop_normal}
  If $P$ is a normal consensus matrix, then $P^*P=PP^*$.
\end{proposition}
\begin{proof}
  Because $P$ is normal, $(I-P)(I-P^\T)=(I-P^\T)(I-P)$. Multiplying $\bm{1}$ from the right, we obtain $(I-P)(I-P^\T)\bm{1}=\bm{0}$, which means $(I-P^\T)\bm{1}\in\ker (I-P)$. Therefore, there exists some real value $a$ that satisfies $(I-P^\T)\bm{1}=a\bm{1}$, because $\ker(I-P)=\vecspan\{\bm{1}\}$. Then, we multiply $\bm{1}^\T$ from the left and obtain $\bm{0}=an$. This means that $a=0$ and $(I-P^\T)\bm{1}=\bm{0}$. Thus, $P$ is doubly-stochastic and the invariant measure $\bm{\pi}$ is $\frac{1}{n}\bm{1}$. Therefore, we obtain $P^*=\correction{P^\T}$ and $P^*P=P^\T P=PP^\T=PP^*$.
\end{proof}

\begin{figure}
  \centering
  \includegraphics[keepaspectratio, scale=0.45]{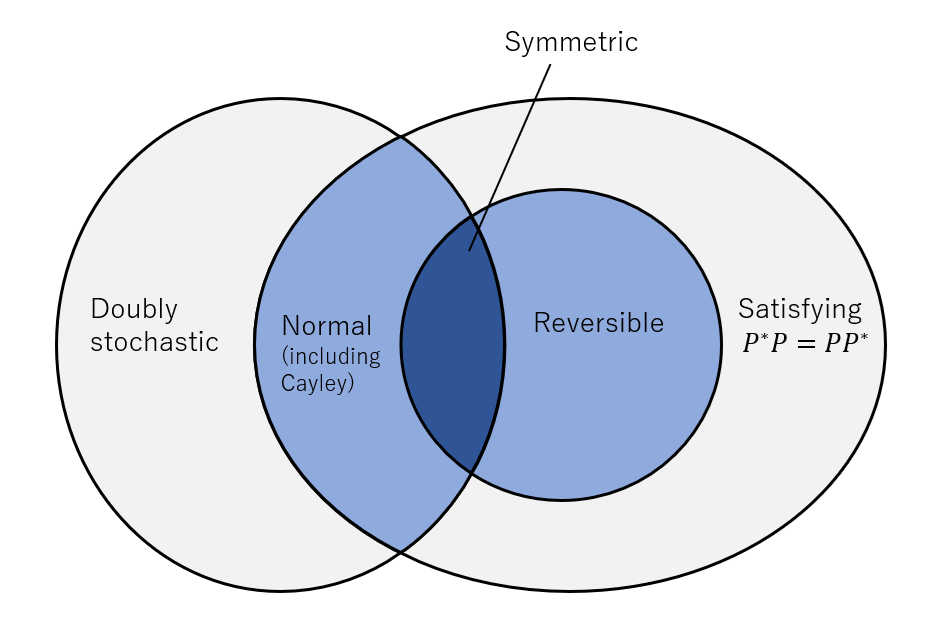}
  \caption{Inclusion properties among classes of consensus matrices.}
  \label{incl}
\end{figure}

Inclusion properties among classes of consensus matrices are shown in \Cref{incl}. Notice that there \correction{is} also a matrix $P$ which is \correction{neither} reversible, nor normal, but satisfies $P^*P=PP^*$, such as the following example \eqref{matrixexample}. Thus, the condition $P^*P=PP^*$ is a \correction{weaker} assumption than reversibility or normality.
\begin{eqnarray}\label{matrixexample}
  P=\frac{1}{2+\sqrt{10}}\begin{pmatrix}
    2 & 1 & -1+\sqrt{10} & 0 \\
    1 & 2 & 0 & -1+\sqrt{10} \\
    0 & 1+\sqrt{10} & 1 & 0 \\
    1+\sqrt{10} & 0 & 0 & 1
  \end{pmatrix}.
\end{eqnarray}

\subsection{A bound for $J(P)$ and $J_{\mathrm{w}}(P)$ using $\averes(C)$}

In this section, we prove the following \Cref{mainres1} that provides upper and lower bounds for $J(P)$ using the effective resistance of $\Phi(P^*P)$. \Cref{mainres1} is a generalization of Theorem 3.1 in \cite{mainp} for reversible $P$.
\begin{theorem}
  \label{mainres1}
  Let $P$ be a consensus matrix with invariant measure $\bm{\pi}$, and let $C_{P^*P}:=\Phi(P^*P)=nP^\T\mathit{\Pi}P$ be the conductance matrix. Then,
  \begin{eqnarray*}
    J(P)\leq\frac{\pi_{\max}^3n^2}{\pi_{\min}}\bar{\mathcal{R}}(C_{P^*P}),\quad
    J_{\mathrm{w}}(P)\leq\pi_{\max}^3n^3\averes(C_{P^*P}),
  \end{eqnarray*}
  where $\pi_{\min}$ and $\pi_{\max}$ are, respectively, the minimum and maximum entries of $\bm{\pi}$. Moreover, if $PP^*=P^*P$, then
  \begin{eqnarray*}
    J(P)\geq\frac{\pi_{\min}^3n^2}{\pi_{\max}}\bar{\mathcal{R}}(C_{P^*P}),\quad
    J_{\mathrm{w}}(P)\geq\pi_{\min}^3n^3\bar{\mathcal{R}}(C_{P^*P}).
  \end{eqnarray*}
\end{theorem}

To prove \Cref{mainres1}, in the same manner as the approach in \cite{mainp}, we define the \textit{weighted average effective resistance}
\begin{eqnarray}\label{weff}
  \averes_{\mathrm{w}}(C):=\frac{1}{2}\bm{\pi}^\T\mathcal{R}(C)\bm{\pi}=\frac{1}{2}\sum_{(u,v)\in V\times V}\mathcal{R}_{uv}(C)\pi_u\pi_v,
\end{eqnarray}
where $\bm{\pi}$ is the invariant measure of reversible $P$ and $C:=\Phi(P)$.

\begin{lemma}\label{rwn}
  For a reversible consensus matrix $P$, let $C:=\Phi(P)$ and $\bm{\pi}$ be the invariant measure of $P$. Then,
  \[n^2\pi_{\min}^2\averes(C)\leq \averes_{\mathrm{w}}(C)\leq n^2\pi_{\max}^2\averes(C).\]
\end{lemma}
\begin{proof}
  Because $\pi_{\min}\leq \pi_u\leq \pi_{\max}$,
  \begin{eqnarray*}
    \averes_{\mathrm{w}}(C)\leq \frac{1}{2}\pi_{\max}^2\sum_{(u,v)\in V\times V}\mathcal{R}_{uv}(C)=n^2\pi_{\max}^2\averes(C),
  \end{eqnarray*}
  and
  \begin{eqnarray*}
    \averes_{\mathrm{w}}(C)\geq \frac{1}{2}\pi_{\min}^2\sum_{(u,v)\in V\times V}\mathcal{R}_{uv}(C)=n^2\pi_{\min}^2\averes(C).
  \end{eqnarray*}
\end{proof}

\begin{lemma}\label{rw}
  For a reversible consensus matrix $P$, let $C:=\Phi(P)$. Then,
  \[\averes_{\mathrm{w}}(C)=\frac{1}{n}\tr G(P),\]
  where $G(P)$ is the Green matrix of $P$ defined in \eqref{defgp}.
\end{lemma}
\begin{proof}
  The similar statement appears in the proof of Lemma 5.8 of \cite{mainp}. By changing $P^2$ to $P$ in \cite{mainp}, we obtain the lemma.
\end{proof}

% \begin{remark}\label{remhitting}
%   $\averes_{\mathrm{w}}(C)$ can be written in terms of hitting time of the Markov chain associated with $P$ (see Appendix A.1).
% \end{remark}

Using $\averes_{\mathrm{w}}(C)$, we prove \Cref{mainres1}.
\begin{proof}[Proof of \Cref{mainres1}]
  First, we show the upper bound of $J_{\mathrm{w}}(P)$ in \Cref{mainres1}. Let $G_{\mathrm{w}}:=\sum_{t\geq 0}(I-\bm{\pi}\bm{1}^\T)(P^\T)^t\mathit{\Pi}P^t(I-\bm{1}\bm{\pi}^\T)$. Because $G_{\mathrm{w}}$ is positive semidefinite, its diagonal elements are nonnegative. Therefore, the matrix $\varPi^{-1} G_{\mathrm{w}}$ satisfies
  \[\pi_{\max}^{-1}(G_{\mathrm{w}})_{ii}\leq (\varPi^{-1} G_{\mathrm{w}})_{ii}\leq \pi_{\min}^{-1}(G_{\mathrm{w}})_{ii}.\]
  By summation from $i=1$ to $n$, we obtain
  \begin{eqnarray}\label{jwpb1}
    \pi_{\max}^{-1}J_{\mathrm{w}}(P)\leq \tr\paren{\varPi^{-1}G_{\mathrm{w}}}\leq \pi_{\min}^{-1}J_{\mathrm{w}}(P),
  \end{eqnarray}
  because $\correction{J_{\mathrm{w}}(P)}=\tr G_{\mathrm{w}}$ by definition. Since
  \begin{eqnarray*}
  \tr\paren{\varPi^{-1}G_{\mathrm{w}}}=\tr\paren{\varPi^{-1}\sum_{t\geq 0}\paren{(P^\T)^t\varPi P^t-\bm{\pi}\bm{\pi}^\T}}=\tr\paren{\sum_{t\geq 0}\paren{(P^*)^tP^t-\bm{1}\bm{\pi}^\T}},
  \end{eqnarray*}
  \eqref{jwpb1} yields
  \begin{eqnarray}\label{jweval}
    \pi_{\max}^{-1}J_{\mathrm{w}}(P)\leq \tr\paren{\sum_{t\geq 0}\paren{(P^*)^tP^t-\bm{1}\bm{\pi}^\T}}\leq \pi_{\min}^{-1}J_{\mathrm{w}}(P).
  \end{eqnarray}
  By \Cref{tracekey}, we obtain
  \begin{eqnarray}\label{jweval2}
    J_{\mathrm{w}}(P) \leq\pi_{\max}\tr\paren{\sum_{t\geq 0}\paren{(P^*)^tP^t-\bm{1}\bm{\pi}^\T}}\leq\pi_{\max}\tr G(P^*P).
  \end{eqnarray}
  Since $P^*P$ is a reversible consensus matrix, Lemmas \ref{rwn} and \ref{rw} imply that
  \[J_{\mathrm{w}}(P)\leq n\pi_{\max}\averes_{\mathrm{w}}(C_{P^*P})\leq n^3\pi_{\max}^3\averes(C_{P^*P}).\]
  Thus, the upper bound of \Cref{mainres1} is obtained.

  Then, we show the lower bound of $J_{\mathrm{w}}(P)$ in \Cref{mainres1}. When $PP^*=P^*P$, we can rewrite (\ref{jweval}) as below:
  \begin{align}\label{jweval3}
    J_{\mathrm{w}}(P)\geq \pi_{\min}\tr\paren{\sum_{t\geq0}\paren{(P^*)^tP^t-\bm{1}\bm{\pi}^\T}}=\pi_{\min}\tr G(P^*P).
  \end{align}
  Lemmas \ref{rwn} and \ref{rw} imply that \[J_{\mathrm{w}}(P)\geq n\pi_{\min}\averes_{\mathrm{w}}(C_{P^*P})\geq n^3\pi_{\min}^3\averes(C_{P^*P}).\] Thus, we obtain the bound of $J_{\mathrm{w}}(P)$.

  Finally, we give a bound $J(P)$ by $J_{\mathrm{w}}(P)$. Comparing (\ref{lqcost}) and (\ref{lqcost2}), we obtain
  \begin{eqnarray}\label{compare12}
    \frac{1}{n\pi_{\max}}J_{\mathrm{w}}(P)\leq J(P)\leq\frac{1}{n\pi_{\min}}J_{\mathrm{w}}(P).
  \end{eqnarray}
  Applying \eqref{jweval2} and \eqref{jweval3} to \eqref{compare12}, the proof is completed.
\end{proof}

% \begin{remark}
%   It seems to be hard to give a lower bound for $\tr \paren{(P^*)^tP^t}$ using $\tr (P^*P)^t$, because $\tr (P^*P)^t$ can be much larger than $\tr (P^*P)^t$. Let us consider the following consensus matrix:
%   \begin{eqnarray*}
%     P_{\varepsilon}=\begin{pmatrix}
%       \varepsilon & 1-\varepsilon & 0 \\
%       0 & \varepsilon & 1-\varepsilon \\
%       \frac{1}{2} & 0 & \frac{1}{2}
%     \end{pmatrix}.
%   \end{eqnarray*}
%   $P^t-\bm{1}\bm{\pi}$ converges exponentially with respect to the second largest eigenvalue of $P$. The eigenvalues of $P_{\varepsilon}$ are $1$ and $\alpha_{\pm}:=-\paren{\frac{1}{4}-\varepsilon}\pm\frac{\mathrm{i}}{2}\sqrt{\frac{7}{4}-2\varepsilon}$. The eigenvalues of $P_{\varepsilon}^*$ are the same due to similarity. Therefore, $(P_{\varepsilon}^*)^tP_{\varepsilon}^t-\bm{1}\bm{\pi}$ converges exponentially to some real matrix with respect to the rate near $1/2$.

%   On the other hand, $P^*P$ has the eigenvalue $1-O(\varepsilon)$. This means that if $\varepsilon$ is small enough, the convergence of $(P^*P)^t$ is very slow, and its trace can be far larger than $\tr (P_{\varepsilon}^*)^tP_{\varepsilon}^t$.
% \end{remark}

\subsection{A bound for $J(P)$ using the graph structure of network}
In this section, we establish bounds for $J(P)$ using the structure of \correction{associated undirected} graph $\mathcal{G}(P)$, the maximum and minimum elements of $\bm{\pi}$, and the minimum element of $P$. Notice that we do not consider the edge weights of $\mathcal{G}(P)$ and focus only on its graph topology of $\mathcal{G}(P)$.

Although \Cref{mainres1} provides bounds when all elements of $P$ are known, \Cref{mainres2} offers bounds even if complete information about the elements of $P$ and $\bm{\pi}$ is unavailable.

\begin{theorem}
  \label{mainres2}
  Let $P$ be a consensus matrix with invariant measure \correction{$\bm{\pi}$}, and let $\mathcal{G}(P)$ be the associated undirected graph with $P$. Then,
  \begin{eqnarray*}    J(P)\leq\frac{\pi_{\max}^3n}{p_{\min}^2\pi_{\min}^2}\averes(\mathcal{G}(P)),\quad
    J_{\mathrm{w}}(P)\leq\frac{\pi_{\max}^3n^2}{p_{\min}^2\pi_{\min}}\averes(\mathcal{G}(P)),
  \end{eqnarray*}
  where $\pi_{\min}$ and $\pi_{\max}$ are respectively the minimum and maximum entries of $\bm{\pi}$, $p_{\min}$ is the minimum nonzero entry of $P$, and $\averes(\mathcal{G}(P))$ is the average effective resistance of a network, associated with $\mathcal{G}(P)$, such that each edge has \correction{unit conductance}. Moreover, if $PP^*=P^*P$, then
  \begin{eqnarray*}
    J(P)\geq\frac{\pi_{\min}^3n}{p_{\max}^2f(\din)\pi_{\max}^2}\averes(\mathcal{G}(P)),\quad
    J_{\mathrm{w}}(P)\geq\frac{\pi_{\min}^3n^2}{p_{\max}^2f(\din)\pi_{\max}}\averes(\mathcal{G}(P)),
  \end{eqnarray*}
  where $p_{\max}$ is the maximum nonzero entry of $P$, and $\din$ is the maximum indegree (excluding self loops) of the associated directed graph $\gdir(P)$, and $f(\din):=4\din^2+2\din-2$.
\end{theorem}

By \Cref{mainres1}, it is enough for the proof to show the following lemma, which is a generalization of Lemma 5.9 in \cite{mainp}.
\begin{lemma}\label{mainlemma2}
  Let $P$ be a consensus matrix with invariant measure $\bm{\pi}$ and $C_{P^*P}:=\Phi(P^*P)=nP^\T\varPi P$. Then,
  \begin{eqnarray*}
    \averes(C_{P^*P})\leq \frac{1}{n\pi_{\min}p_{\min}^2}\averes(\mathcal{G}(P)).
  \end{eqnarray*}
  Moreover, if $PP^*=P^*P$, then
  \begin{eqnarray*}
    \averes(C_{P^*P})\geq \frac{1}{n\pi_{\max}f(\din) p_{\max}^2}\averes(\mathcal{G}(P)).
  \end{eqnarray*}
\end{lemma}
\color{black}

Since the associated graph of the consensus matrix $P$ can be directed, we evaluate $\averes(C_{P^*P})$ using the maximum indegree $\din$ of the graph nodes of $\gdir(P)$, unlike previous research \cite{mainp}.

We employ the undirected graph $\mathcal{G}(P)$ instead of the directed graph \correction{$\gdir(P)$}, because the following property is essential for proving \Cref{mainres2}.
\begin{lemma}[Rayleigh's monotonicity law]\label{rayleigh}
  Let $C$ and $C'$ be conductance matrices such that $C\leq C'$. Then, $\mathcal{R}(C)\geq \mathcal{R}(C')$.
\end{lemma}
In plain words, if the resistance of any edge increases, then the resistance between any pair of nodes also increases. The proof can be found in \cite{randomwalks}.

To prove \Cref{mainlemma2}, we first compare $\averes(C_{P^*P})$ and $\averes(\mathcal{G}(P^*P))$ using \Cref{rayleigh}, and then compare $\averes(\mathcal{G}(P^*P))$ and $\averes(\mathcal{G}(P))$. Recall that $\averes(\mathcal{G}(P^*P))$ \correction{depends only on} the graph topology of $\mathcal{G}(P^*P)$ and is independent of each element of $P^*P$. In contrast, $\averes(C_{P^*P})$ is affected by each element of $C_{P^*P}$, which can \correction{differ} to each other.

\begin{lemma}\label{simple_monotonicity}
Let $\mathcal{G}(P^*P)$ be the undirected graph associated with $P^*P$. Then, \begin{eqnarray}\label{simple_monotonicity_ineq}
    \frac{1}{n\pi_{\max}(\din+1) p_{\max}^2}\averes(\mathcal{G}(P^*P))\leq \averes(C_{P^*P})\leq \frac{1}{n\pi_{\min}p_{\min}^2}\averes(\mathcal{G}(P^*P)).
  \end{eqnarray}
\end{lemma}

\begin{proof}
We consider the range of the elements of $C_{P^*P}$. When $(C_{P^*P})_{uv}\neq 0$, there exists $w$ such that $P_{wu}$ and $P_{wv}$ are positive, and hence $(C_{P^*P})_{uv}=n\sum_w\pi_wP_{wu}P_{wv}\geq n\pi_{\min}p_{\min}^2$. Moreover, for any fixed $u$, the number of nonzero entries of $P_{wu}$  is at most $\din+1$. Therefore, we obtain $(C_{P^*P})_{uv}\leq n\pi_{\max}(\din+1) p_{\max}^2$. By applying \Cref{rayleigh} to $(C_{P^*P})_{uv}$, we obtain \eqref{simple_monotonicity_ineq}.
\end{proof}
\color{black}

Next, we evaluate $\averes(\mathcal{G}(P^*P))$ using $\averes(\mathcal{G}(P))$, both of which are determined by the positions of the nonzero elements in $P$. Because $(P^*P)_{uv}=\sum_w\pi_u^{-1}\pi_wP_{wu}P_{wv}$, the entry $(P^*P)_{uv}$ is positive if and only if there exists $w$ such that both $P_{wu}$ and $P_{wv}$ are positive. From a \correction{graph-theoretic} perspective, this means that $\mathcal{G}(P^*P)$ has an edge $\{u,v\}$ if and only if there exists a node $w$ with edges $(w,u)$ and $(w,v)$ in $\gdir(P)$. 

To compare the places of edges in $\mathcal{G}(P^*P)$ and $\mathcal{G}(P)$, we introduce the novel concepts of a \textit{back-and-forth path} and its \textit{pivot node}. A \textit{back-and-forth path} between $u$ and $v$ is a path which consists of $\{u,w\},\{w,v\}\in\mathcal{G}(P)$ where $(w,u)$ and $(w,v)$ are edges in $\gdir(P)$. We refer to $w$ as the \textit{pivot node} of the path (see \Cref{newedge}). When $\gdir(P)$ is symmetric, namely, $P_{ij}\neq 0$ if and only if $P_{ji}\neq 0$, back-and-forth paths reduce to paths of length $2$ in $\mathcal{G}(P)$. In such cases, including when $P$ is reversible, $\mathcal{G}(P^*P)$ has the same graph topology as $\mathcal{G}(P^2)$. In general, the graph $\mathcal{G}(P^h)$ is called $h$-fuzz, and previous research has shown the relationship between $\averes(\mathcal{G}(P))$ and $\averes(\mathcal{G}(P^h))$ \cite{hfuzz}. However, for nonreversible $P$, it is worthwhile to compare $\mathcal{G}(P)$ and $\mathcal{G}(P^*P)$. Notice that the concept of a pivot node is close to that of a common in-neighbor, but differs slightly. Here, we emphasize that both the pivot node and the back-and-forth are concepts on the undirected graph, although they are defined via conditions on the directed graph.

Using the concept of a back-and-forth path, we compare $\averes(\mathcal{G}(P^*P))$ and $\averes(\mathcal{G}(P))$ in the following lemma.
Before stating the lemma, we recall that throughout the paper we assume that all diagonal entries of $P$ are positive. 
This ensures that every node in $\gdir(P)$ has a self-loop, which is used in the proof below.

\begin{lemma}\label{bounds_GPP}
  Let $\mathcal{G}(P)$ and $\mathcal{G}(P^*P)$ be the associated undirected graph with $P$ and $P^*P$, respectively. Then,
  \begin{eqnarray*}
    \frac{1}{4\dout-2}\mathcal{R}_{uv}\paren{\mathcal{G}(P)}\leq \mathcal{R}_{uv}\paren{\mathcal{G}(P^*P)}\leq\mathcal{R}_{uv}\paren{\mathcal{G}(P)},
  \end{eqnarray*}
  where $\dout$ is the maximum outdegree (excluding self loops) of the graph nodes of $\gdir(P)$.
\end{lemma}

\begin{proof}
The right-hand inequality is straightforward. For each edge $(u,v)$ in $\gdir(P)$, we can choose a back-and-forth path between $u$ and $v$, because there are edges $(u,u)$ (a self loop) and $(u,v)$ in $\gdir(P)$. Therefore, any edge appearing in $\mathcal{G}(P)$ also appears in $\mathcal{G}(P^*P)$. 
In fact,
\begin{align*}
    (P^*P)_{uv}
=\frac{1}{\pi_u}\sum_{w}\pi_w P_{wu}P_{wv}
\geq
\frac{1}{\pi_u}\pi_u P_{uu}P_{uv}
= P_{uu}P_{uv} >0.
\end{align*}
By Lemma~\ref{rayleigh}, $\mathcal{R}_{uv}(\mathcal{G}(P^*P))\leq \mathcal{R}_{uv}(\mathcal{G}(P))$ for all $u,v$.

For the left-hand inequality, we adapt the standard argument for fuzz graphs \cite[Lem.~5.5.1]{hfuzz}. 
We write $\mathcal{G}(P)=(V,\mathcal{E})$ and $\mathcal{G}(P^*P)=(V,\mathcal{E}^*)$, and call the edges in $\mathcal{E}^*\setminus\mathcal{E}$ ``new edges" (see \Cref{newedge}). Consider replacing each new edge in $\mathcal{G}(P^*P)$ with a series of two edges and an intermediate node, resulting in a modified graph $\bar{\mathcal{G}}=(\bar{V},\bar{\mathcal{E}})$. In this transformation, the resistances of the new edges are doubled. By Rayleigh's monotonicity law, the new resistance of every pair of two nodes in $V$ cannot exceed twice the original value, that is, $R_{uv}(\bar{\mathcal{G}})\leq 2R_{uv}(\mathcal{G}(P^*P))$ for every pair of nodes $u,v\in V$.

Next, focus on an intermediate node $x_{uv}$ added when replacing a new edge $\{u,v\}$. For every new edge $\{u,v\}$ in $\mathcal{G}(P^*P)$, there exists a back-and-forth path in $\mathcal{G}(P)$. Let $w$ be its pivot node. We then add a zero-resistance edge between $x_{uv}$ and $w$, and regard $x_{uv}$ and $w$ as an identical node. After this operation, the resulting graph $\mathcal{G}'$ satisfies $\mathcal{R}_{uv}(\mathcal{G}')\le \mathcal{R}_{uv}(\bar{\mathcal{G}})$ for every pair of $u,v\in V$ by Rayleigh’s monotonicity (see \Cref{Fig:evaluation})\footnote{Intuitively, this operation can be understood as short-circuiting.}.
\color{black} The graph $\mathcal{G}'$ differs from $\mathcal{G}(P)$ only in the multiplicity of edges, so we compare the two by counting parallel edges. 
Note that $\mathcal{G}'$ has one self-loop at each node, as does $\mathcal{G}(P)$. 
For $u,v\in V$ with $u\neq v$, the number of edges between $u$ and $v$ in $\mathcal{G}'$ is bounded by the number of back-and-forth paths through $\{u,v\}$ in $\mathcal{G}(P)$ plus $1$, the original edge.
Let $\eta$ be an upper bound on the number of such paths plus the original edge. 
Since $k$ parallel unit-resistance edges are equivalent to a single edge of resistance $1/k$, Rayleigh’s monotonicity law yields
$
\mathcal{R}_{uv}\left(\mathcal{G}'\right) \geq \frac{1}{\eta}\,\mathcal{R}_{uv}\left(\mathcal{G}(P)\right)$.
Combining this with the previous inequality, we obtain
\begin{align*}
\frac{1}{\eta}\,\mathcal{R}_{uv}\left(\mathcal{G}(P)\right) \leq \mathcal{R}_{uv}\left(\mathcal{G}'\right) \leq \mathcal{R}_{uv}\left(\bar{\mathcal{G}}\right) \leq 2\mathcal{R}_{uv}\left(\mathcal{G}(P^*P)\right)    
\end{align*}
for all $u,v\in V$.

\begin{figure}
  \centering
  \includegraphics[keepaspectratio, scale=0.5]{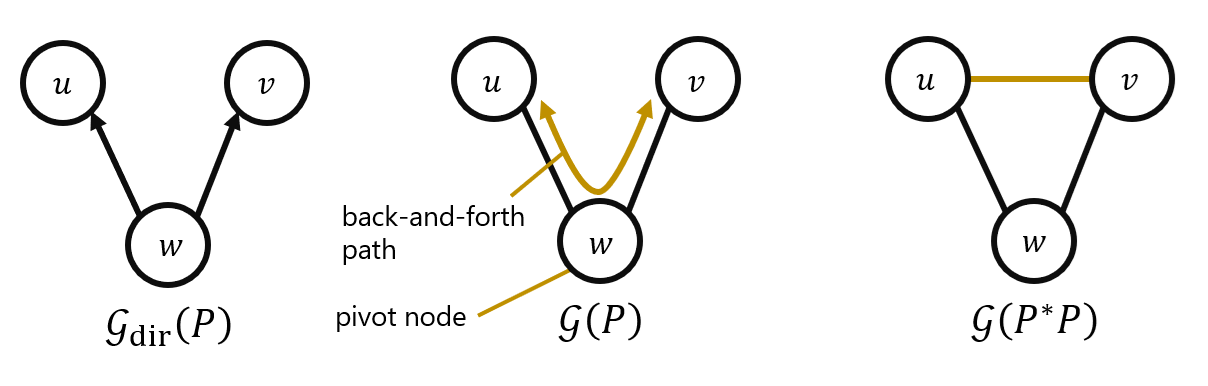}
  \caption{Notions of back-and-forth path, pivot, and new edge.}
  \label{newedge}
\end{figure}

\begin{figure}
  \centering
  \includegraphics[keepaspectratio, scale=0.6]{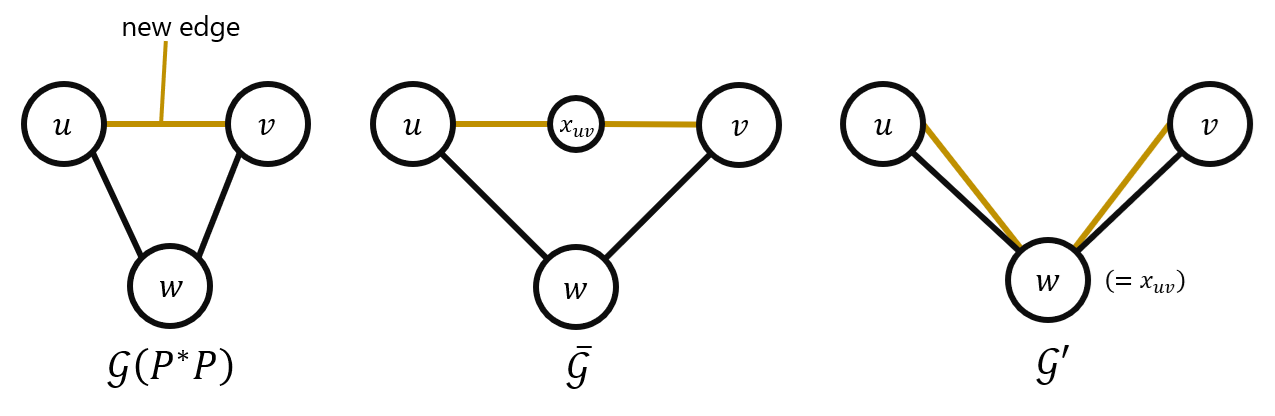}
  \caption{Graphs used to evaluate $\averes(\mathcal{G}(P^*P))$.}
  \label{Fig:evaluation}
\end{figure}

Finally, we bound $\eta$. 
For a back-and-forth path through $\{u,v\}$ in $\mathcal{G}(P)$, the pivot is either $u$ or $v$. 
If the pivot is $u$, then the other endpoint (distinct from $v$) can be chosen among at most $\dout-1$ out-neighbors of $u$; the same bound holds when the pivot is $v$. 
Including the original edge $\{u,v\}$ of $\bar{\mathcal{G}}$, we obtain $\eta\le 2(\dout-1)+1=2\dout-1$, and hence
$\frac{1}{2\dout-1}\,\mathcal{R}_{uv}\left(\mathcal{G}(P)\right)\ \le\ 2\,\mathcal{R}_{uv}\left(\mathcal{G}(P^*P)\right)$,
which proves the desired lower bound for $\averes\!\left(\mathcal{G}(P^*P)\right)$.
\end{proof}

Applying Lemma~\ref{bounds_GPP} to $P^*$, we obtain the following corollary.

\begin{corollary}\label{bounds_cor}
    If $P^*P=PP^*$, then
    \begin{eqnarray*}
        \frac{1}{4\din-2}\mathcal{R}_{uv}\paren{\mathcal{G}(P)}\leq \mathcal{R}_{uv}\paren{\mathcal{G}(P^*P)}\leq\mathcal{R}_{uv}\paren{\mathcal{G}(P)},
    \end{eqnarray*}
    under the same assumption as \Cref{bounds_GPP}.
\end{corollary}

\begin{proof}
    Substituting $P$ in \Cref{bounds_GPP} \correctionx{with} $P^*$, then we obtain
    \begin{eqnarray*}
        \frac{1}{4\din-2}\mathcal{R}_{uv}\paren{\mathcal{G}(P^*)}\leq \mathcal{R}_{uv}\paren{\mathcal{G}(PP^*)}\leq\mathcal{R}_{uv}\paren{\mathcal{G}(P^*)},
    \end{eqnarray*}
    since $(P^*)^*=P$ and the outdegrees of $\gdir(P^*)$ corresponds to the indegrees of $\gdir(P)$. By $P^*P=PP^*$ and $\mathcal{G}(P^*)=\mathcal{G}(P)$, the proof is completed.
\end{proof}

We are now in a position to prove \Cref{mainlemma2}.
\begin{proof}[Proof of \Cref{mainlemma2}]
    Combining \Cref{simple_monotonicity} and \Cref{bounds_GPP}, we obtain
    \begin{eqnarray*}
        \frac{1}{n\pi_{\max}(4\dout-2)(\din+1) p_{\max}^2}\averes(\mathcal{G}(P^*P))\leq \averes(C_{P^*P})\leq \frac{1}{n\pi_{\min}p_{\min}^2}\averes(\mathcal{G}(P^*P)).
    \end{eqnarray*}
    In addition, if $P^*P=PP^*$, we can change $\dout$ in the leftmost term to $\din$ by \Cref{bounds_cor}. In this case, the lower bound for $\averes(C_{P^*P})$ can be obtained by $\frac{1}{n\pi_{\max}f(\din) p_{\max}^2}\averes(\mathcal{G}(P^*P))$ where $f(\delta)=4\delta^2+2\delta-2$. Using \Cref{mainres1}, the proof is completed.
\end{proof}

\begin{proof}[Proof of \Cref{mainres2}]
Follows immediately from \Cref{mainres1} and \Cref{mainlemma2}.
\end{proof}
%ここだけ地の文で「定理3.6は定理3.3と補題3.7から即座に導出される」と書くのは構造化されていないため、短いが証明の形式として書いた

In the end of the section, we mention a special case in which our theorem can be written more simply.

\begin{corollary}\label{cayleycorollary}
    Let $P$ be a normal consensus matrix with invariant measure $\bm{\pi}$, and let $\mathcal{G}(P)$ be the associated undirected graph with $P$. Then,
    \begin{eqnarray*}
        \frac{1}{p_{\max}^2f(\din)}\averes(\mathcal{G}(P))\leq J(P)\leq \frac{1}{p_{\min}^2}\averes(\mathcal{G}(P)),
    \end{eqnarray*}
    where $f(\din)=4\din^2+2\din-2$.
\end{corollary}
\begin{proof}
    Because a normal stochastic matrix is doubly stochastic, $\bm{\pi}=\frac{1}{n}\bm{1}$ and $\pi_{\min}=\pi_{\max}=\frac{1}{n}$. In addition, \correctionx{by \Cref{Prop_normal}}, we have $P^*P=PP^*$. Using \Cref{mainres2}, we complete the proof.
\end{proof}
\color{black}

\section{Important examples and applications}
In this section, we will consider important examples for applying the result in Section 3.

\subsection{Necessity of the assumption $P^*P=PP^*$}
We have assumed the equality $P^*P=PP^*$ for the lower bounds in \Cref{mainres1} and \Cref{mainres2}. We will discuss the necessity of the assumption. Consider the following consensus matrix:
\begin{eqnarray*}
  P_{\varepsilon}=\begin{pmatrix}
    \varepsilon & 1-\varepsilon & 0 \\
    0 & \varepsilon & 1-\varepsilon \\
    \frac{1}{2} & 0 & \frac{1}{2}
  \end{pmatrix}.
\end{eqnarray*}
We enforce $\varepsilon>0$ to comply with our standing assumption that the diagonal entries are positive; otherwise some of the lemmas used in the paper do not apply.

We compute $J(P_{\varepsilon})$ and the corresponding upper and lower bounds from \Cref{mainres1,mainres2} for $\varepsilon\le 1/2$. 
The results are shown in \Cref{figure:smallcase}.
In \Cref{figure:smallcase}, the lower bound from \Cref{mainres1} fails to hold when $\varepsilon<1/2$. 
This is unsurprising, since $P_\varepsilon^*P_\varepsilon\neq P_\varepsilon P_\varepsilon^*$ for $\varepsilon<1/2$, and the proof of \Cref{mainres1} relies on the commutation $P^*P=PP^*$.

To understand the mechanism, note that, as $\varepsilon\to 0$, $\pe^t-\bm{1}\bm{\pi}$ converges exponentially at a rate determined by the second largest eigenvalue of $P$. The eigenvalues of $\pe$ are $1$ and $\alpha_{\pm}:=-\paren{\frac{1}{4}-\varepsilon}\pm\frac{\mathrm{i}}{2}\sqrt{\frac{7}{4}-2\varepsilon}$. The eigenvalues of $\pe^*$ coincide with those of $\pe$ by similarity. Therefore, $(\pe^*)^t\pe^t-\bm{1}\bm{\pi}$ converges exponentially to some real matrix at a rate close to $1/2$. On the other hand, $\pe^*\pe$ has the second eigenvalue $1-O(\varepsilon)$. This means that if $\varepsilon$ is small enough, the convergence of $(\pe^*\pe)^t$ is very slow, and $\tr (\pe^*\pe)^t$ can be much larger than $\tr (\pe^*)^t\pe^t$. Therefore, the right-hand side of \eqref{jweval3} becomes much larger than the leftmost hand, and \Cref{mainres1} does not hold.
\color{black}

However, in \Cref{figure:smallcase}, it seems that the lower bound of \Cref{mainres2} holds for small $\varepsilon$, although there are no theoretical guarantees in our paper. This is likely due to the rough evaluation of \Cref{mainlemma2}. When $\varepsilon$ decreases, the average effective resistance $\averes(C_{P^*P})$ increases linearly to $1/\varepsilon$, \correctionx{making the evaluation in \Cref{mainlemma2} increasingly loose}. Taking into account this behavior, it may be possible to find sharper or more general lower bounds for $J(P)$ in future research.

\begin{figure}[htbp]
  \begin{minipage}[htbp]{0.48\linewidth}
    \centering
    \includegraphics[keepaspectratio, scale=0.38]{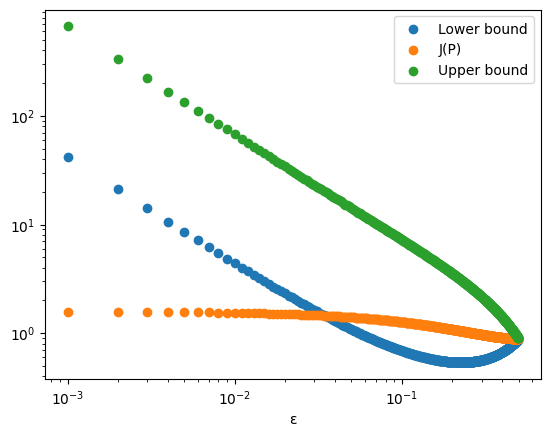}
  \end{minipage}
  \begin{minipage}[htbp]{0.48\linewidth}
    \centering
    \includegraphics[keepaspectratio, scale=0.38]{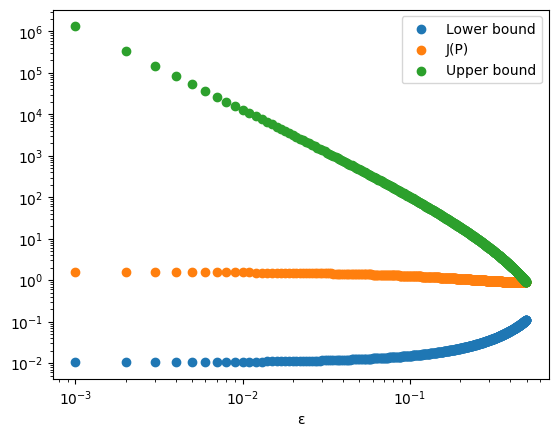}
  \end{minipage}
  \caption{The values of $J(P_{\varepsilon})$ for $\varepsilon\in[0.001,0.5]$. The left panel shows the lower and upper bounds in \Cref{mainres1}. The right panel shows the lower and upper bounds in \Cref{mainres2}.}
  \label{figure:smallcase}
\end{figure}

\subsection{Applications to Cayley graphs}
An important application where the property $P^*P=PP^*$ holds is that of Cayley consensus matrices. Given a finite abelian group $(V,+)$ and a subset $S\subseteq V$, the Cayley graph $\mathcal{G}$ is defined as a directed graph with the vertex set $V$ and the edge set $E=\{(u,v)\mid u-v\in S\}$. The associated Cayley matrix $P\in \mathbb{R}^{V\times V}$ on this graph is defined by $P_{uv}=g(u-v)$ for some function $g$, called a generator,
where rows and columns are indexed by the elements of $V$. 
If a consensus matrix $P$ is a Cayley matrix, then $P^*P=PP^*$ holds, since Cayley matrices are normal (see \cite[Section III]{cayley}) and \Cref{Prop_normal} applies.

\begin{figure}
\centering
\includegraphics[keepaspectratio, scale=0.6]{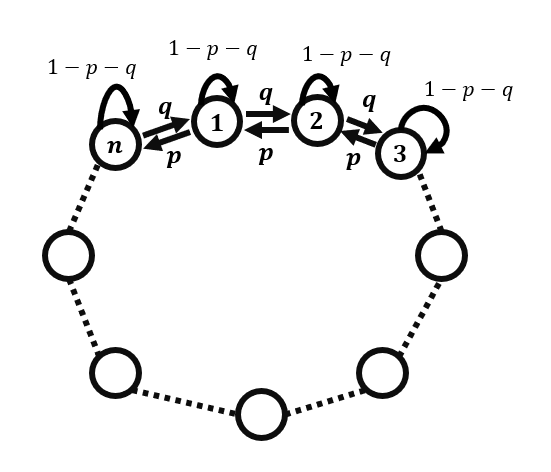}
\caption{Consensus algorithm on a circle.}
\label{fig:circlegraph}
\end{figure}

A Cayley consensus matrix naturally arises for consensus dynamics on periodic grids (discrete tori). We begin with the one-dimensional case. As shown in \Cref{fig:circlegraph}, there are $n$ agents on a circle, and each agent updates its information by the weighted sum of the left agent (with weight $p$), the right agent (with $q$) and itself (with weight $1-p-q$). This graph is the Cayley graph generated by the cyclic group 
$\mathbb{Z}_n := \{0,1,\ldots,n-1\}$ with addition modulo $n$ 
and the generating set $S=\{-1,0,1\}$. The consensus matrix is
\begin{eqnarray*}P=
  \begin{pmatrix}
    1-p-q & q & 0 & \cdots & p \\
    p & 1-p-q & q & \cdots & 0 \\
    0 & p & 1-p-q & \cdots & 0 \\
    \vdots & \vdots & \vdots & \ddots & \vdots \\
    q & 0 & 0 & \cdots & 1-p-q \\
  \end{pmatrix}.
\end{eqnarray*} 
This is a Cayley matrix obtained by
\begin{eqnarray*}
  g(x):=\begin{cases}
    p & \text{if }x\equiv 1\mod n\\
    1-p-q & \text{if }x\equiv 0\mod n\\
    q & \text{if }x\equiv -1\mod n\\
    0 & \text{otherwise}.
  \end{cases}
\end{eqnarray*}

Higher-dimensional tori are treated similarly using $\mathbb{Z}_n^d$. If we define $S$ as the neighbor set (for example, $\{-1,0,1\}^d$) and the value of $g(\bm{h})$ for each $\bm{h}\in S$ so that the sum of them is $1$, then we can consider a consensus algorithm on a $d$ dimensional grid where each agent communicates only with the neighbor agents. The example of a two dimensional torus is shown in \Cref{figure:gridgraph}. In this case, we set $S=\{-1,0,1\}^2$ and $g(\pm1,\pm1)=0$. Grids with boundaries can be reduced to the torus case via an appropriate reflection argument \cite{cayley}.

\begin{figure}
  \centering
  \includegraphics[keepaspectratio, scale=0.5]{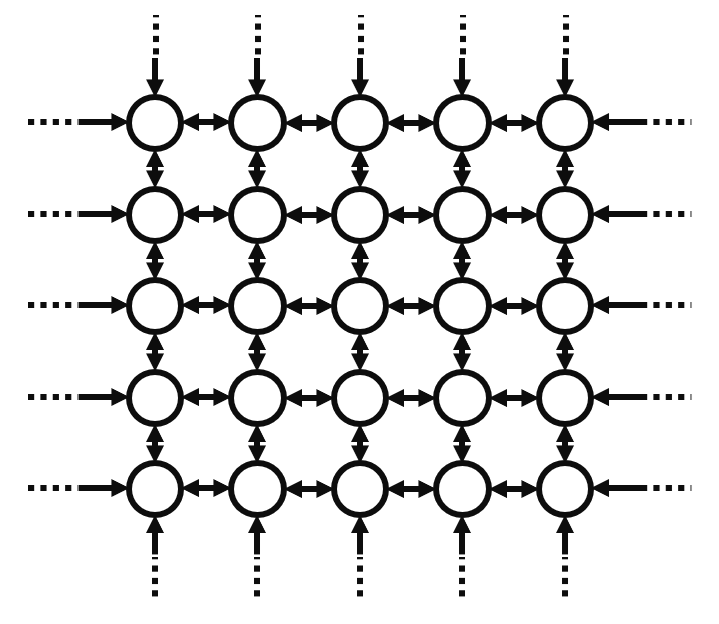}
  \caption{A grid graph on two dimensional torus. The nodes in the leftmost column and the top row are adjacent to those in the rightmost column and in the bottom row, respectively.}
  \label{figure:gridgraph}
\end{figure}

We first recall the asymptotic behavior of $J(P)$ for Cayley consensus matrices.

\begin{proposition}[Proposition 1 in \cite{cayley}, modified]\label{asymptotic1}
  Let $\{P_n\}_{n\geq 1}$ be a Cayley consensus matrix family on the Cayley graph on $\mathbb{Z}_n^d$ with the neighbor set $\{-1,0,1\}^d$, associated with a fixed generator $g$. Then, there exist $C_d,C_d'>0$ (depending only on $d$) such that:
  \begin{itemize}
    \item if $d=1$:
    \[C_dn^d\leq J(P_n)\leq C_d'n^d\]
    \item if $d=2$:
    \[C_d\log(n^d)\leq J(P_n)\leq C_d'\log(n^d)\]
    \item if $d\geq 3$:
    \[C_d\leq J(P_n)\leq C_d'\]
  \end{itemize}
\end{proposition}

This proposition says that, for the Cayley matrix $P\in \mathbb{R}^{V\times V}$ on the grid graph on the $d$ dimensional torus, \[J(P)=\begin{cases}
    \Theta(|V|) & \text{if }d=1\\
    \Theta(\log|V|) & \text{if }d=2\\
    \Theta(1) & \text{if }d\geq 3.
\end{cases}\]

Since $J(P)$ is already understood, we next relate $J(P)$ to our upper and lower bounds involving $\averes(\mathcal{G}(P))$. \Cref{cayleycorollary} establishes a direct relationship between 
the LQ cost $J(P)$ and the average effective resistance $\averes(\mathcal{G}(P))$ 
for normal consensus matrices. 
To further clarify the behavior of $\averes(\mathcal{G}(P))$ in concrete settings, 
we specialize to Cayley graphs on $\mathbb{Z}_n^d$ and derive its asymptotic scaling.

\begin{corollary}\label{cayleyaverescorollary}
    Let $\mathcal{G}_n$ be the Cayley graph on $\mathbb{Z}_n^d$ with the neighbor set $\{-1,0,1\}^d$. Then, there exist $C_d,C_d'>0$ (depending only on $d$) such that:
  \begin{itemize}
    \item if $d=1$:
    \[C_dn^d\leq \averes(\mathcal{G}_{\correctionx{n}})\leq C_d'n^d\]
    \item if $d=2$:
    \[C_d\log(n^d)\leq \averes(\mathcal{G}_{\correctionx{n}})\leq C_d'\log(n^d)\]
    \item if $d\geq 3$:
    \[C_d\leq \averes(\mathcal{G}_{\correctionx{n}})\leq C_d'\]
  \end{itemize}
\end{corollary}

\begin{proof}
    Let $\{P_n\}_{n\geq 1}$ be a Cayley consensus matrix family on $\mathcal{G}_n$, associated with a fixed generator $g$. Then, $p_{\min},p_{\max}$ and an upper bound of $f(\din)$ is determined by $g$. By \Cref{cayleycorollary}, the upper and lower bounds of $J(P_n)$ \correctionx{depend linearly on} $\averes(\mathcal{G}_n)$ under fixed $g$ and $d$. Therefore, substituting $J(P_n)$ \correctionx{with} $\averes(\mathcal{G}_n)$ in \Cref{asymptotic1}, the proof is completed.
\end{proof}

\correctionx{Thus, not only $J(P)$ but also the bounds of $J(P)$ in \Cref{cayleycorollary} grow logarithmically in $d=2$ and keep constant in $d=3$. To check these behaviors, we calculate $J(P)$ and the lower and upper bounds in \Cref{cayleycorollary}} as follows:
\begin{itemize}
    \item We fix the dimension $d=2$ or $3$ and the size of the grid $n$. Notice that we will generate the graph with $n^d$ nodes.
    \item For each $d$, we define the generator $f$ in the following two patterns:
    \begin{itemize}
        \item \textbf{Case 1.} We set $g(\bm{h})\in (0,1]$ for each $\bm{h}\in\{-1,0,1\}^d$ randomly and normalize so that the sum of $g(\bm{h})$ is $1$. Then, we check whether $g(\bm{h})$ is in a certain range $[p_{\min},p_{\max}]$. Here we define $p_{\min}=0.05$ and $p_{\max}=0.2$ in the case $d=2$, and $p_{\min}=0.01$ and $p_{\max}=0.1$ in the case $d=3$. In this case, the maximum indegree $\din$ of the graph is $8$ (in $d=2$) or $26$ (in $d=3$).
        \item \textbf{Case 2.} We set $g(\bm{h})=\frac{1}{d+1}$ for each $\bm{h}\in\{\bm{0},\mathbf{e}_1,\dots,\mathbf{e}_d\}$. In this case, the maximum indegree $\din$ of the graph is $2$ (in $d=2$) or $3$ (in $d=3$).
    \end{itemize}
    \item Generate the Cayley graph $\mathcal{G}$ and the Cayley matrix \correction{$P\in\correctionx{\mathbb{R}}^{n^d\times n^d}$} by the generator $g$. In case 1, $P$ is randomly determined, so we generate $20$ instances of $g$ for each $d,n$.
    \item Calculate $\averes(\mathcal{G})$ and $J(P)$. Compare the lower and upper bounds of \Cref{cayleycorollary} and $J(P)$.
\end{itemize}

\begin{figure}[htbp]
  \begin{minipage}[htbp]{0.48\linewidth}
    \centering
    \includegraphics[keepaspectratio, scale=0.34]{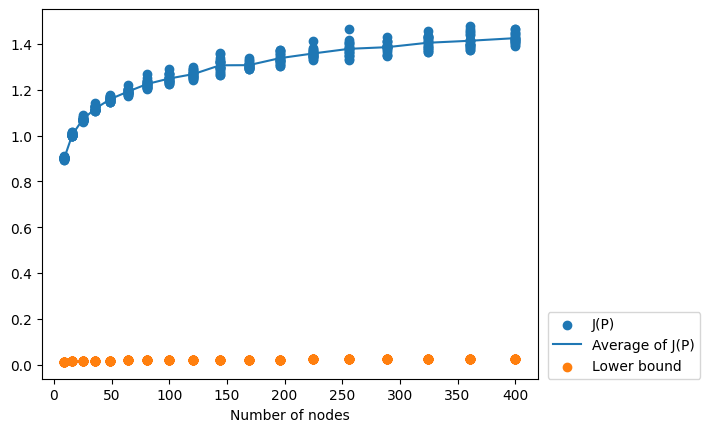}
  \end{minipage}
  \begin{minipage}[htbp]{0.48\linewidth}
    \centering
    \includegraphics[keepaspectratio, scale=0.34]{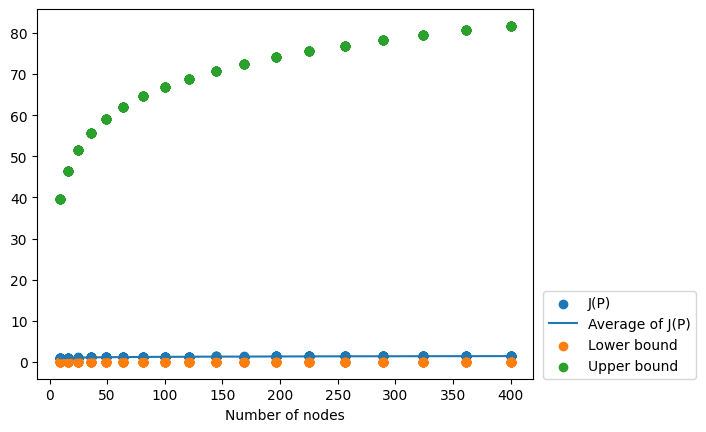}
  \end{minipage}
  \caption{The values of $J(P)$ in $d=2$, where $g$ is generated by the case 1. The left and right panels show the theoretical lower bound in \Cref{cayleycorollary}. The right panel also shows the upper bound, which is too large to draw in the left.}
  \label{figure:cayleyd2}
\end{figure}

\begin{figure}[htbp]
  \begin{minipage}[htbp]{0.48\linewidth}
    \centering
    \includegraphics[keepaspectratio, scale=0.34]{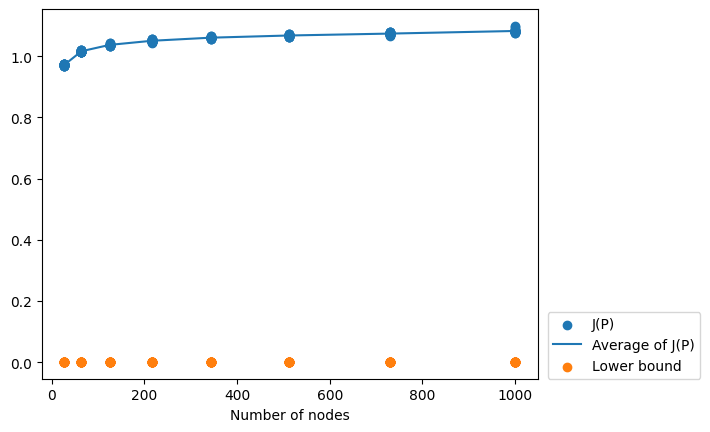}
  \end{minipage}
  \begin{minipage}[htbp]{0.48\linewidth}
    \centering
    \includegraphics[keepaspectratio, scale=0.34]{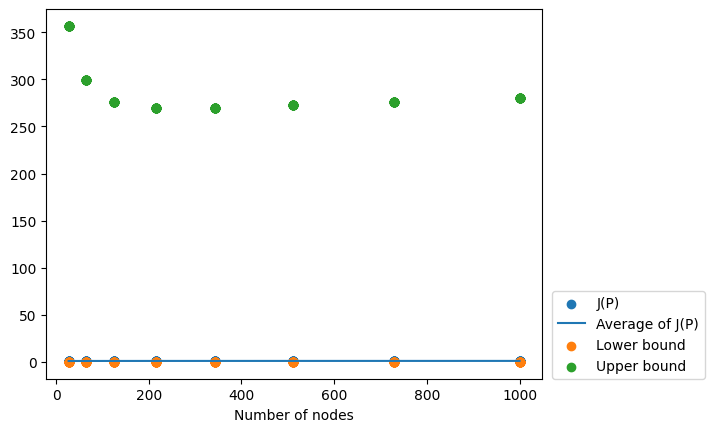}
  \end{minipage}
  \caption{The values of $J(P)$ in $d=3$, where $g$ is generated by the case 1. The left and right panels show the theoretical lower bound in \Cref{cayleycorollary}. The right panel also shows the upper bound.}
  \label{figure:cayleyd3}
\end{figure}

\begin{figure}[htbp]
  \begin{minipage}[htbp]{0.48\linewidth}
    \centering
    \includegraphics[keepaspectratio, scale=0.36]{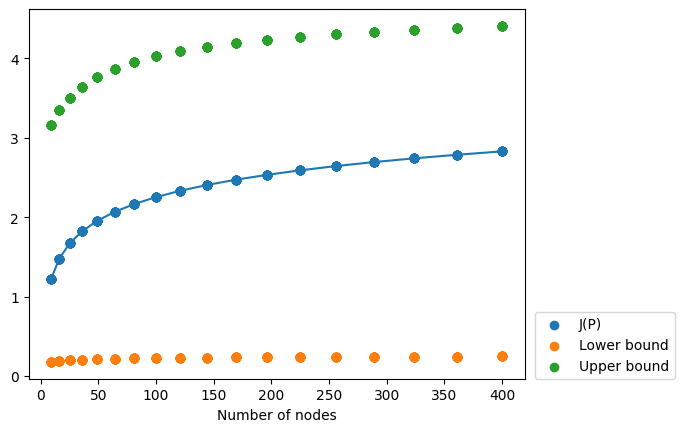}
  \end{minipage}
  \begin{minipage}[htbp]{0.48\linewidth}
    \centering
    \includegraphics[keepaspectratio, scale=0.36]{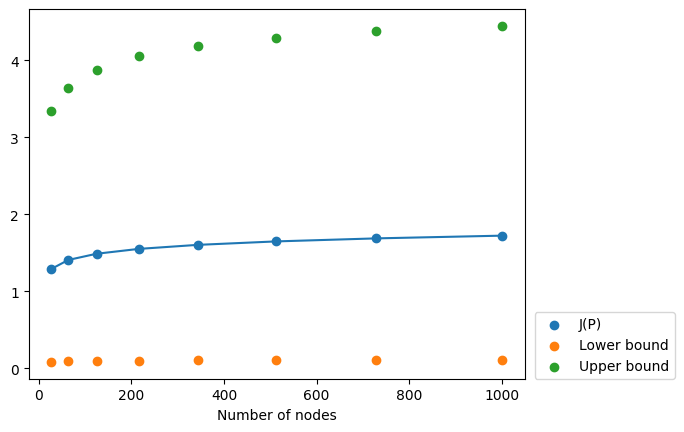}
  \end{minipage}
  \caption{The values of $J(P)$ in \correctionx{$d=2$ (left panel) and $d=3$ (right panel)}, where $g$ is generated by the case 2. The left and right panels show the theoretical lower and upper bounds in \Cref{cayleycorollary}.}
  \label{figure:cayleydir}
\end{figure}

The result is shown in \Cref{figure:cayleyd2}, \ref{figure:cayleyd3}, and \ref{figure:cayleydir}. In these figures, $J(P)$ appears to increase logarithmically in $d=2$, and does not change much in $d=3$. In addition, the change of $\averes(\mathcal{G})$ seems to be consistent with \Cref{cayleyaverescorollary}. These results also imply that the bound we show in the paper is not sharp in some cases. In fact, the ratio of \correctionx{the lower bound to the upper bound} in the case 1 of $d=3$ is $f(26)\times 0.1^2/0.01^2=275400$. However, in the case 2 of $d=2$, all elements are the same value, so the ratio is $f(2)=18$. Thus, our bounds are most informative when the in-degree is small and the generator weights are nearly uniform.

\subsection{Application to geometric graphs} \label{subsection_geometric}
\color{black}
In this section, we apply the \correction{bounds} obtained in \Cref{mainres2} to geometric graphs. Geometric graphs model the position in a $d$-dimensional space, so the result can be useful for controlling objects in the real world. This is a generalization of the result shown in \cite{mainp}.

A geometric graph is a connected, undirected and unweighted graph $\mathcal{G}(V,\mathcal{E})$ such that $V\subseteq \mathcal{Q}$, where $\mathcal{Q}:=[0,l]^d\subseteq \mathbb{R}^d$ and $|V|=n$. Notice that there are no constraint on $\mathcal{E}$ at first, but the distances between connected pairs of two nodes are used to define parameters.

For such graphs, we can define the following parameters \cite{mainp}:
\begin{enumerate}
  \item the minimum node distance \[s=\underset{u,v\in V,\ u\neq v}{\min}d_{\mathrm{E}}(u,v),\]where $d_{\mathrm{E}}$ denotes Euclidean distance;
  \item the maximum connected range \[r=\underset{(u,v)\in \mathcal{E}}{\max}d_{\mathrm{E}}(u,v),\]
  \item the maximum uncovered \correction{radius} \[\gamma=\sup\parbr{r>0 \mid \exists\bm{x}\in \mathcal{Q},\ B(\bm{x},r)\cap V=\emptyset\in \mathcal{E}},\] where $B(\bm{x},r)$ is a $d$-dimensional ball centered in $\bm{x}\in\mathbb{R}^d$ and with radius $r$;
  \item the minumum ratio between Euclidean distance and graphical distance \[\rho=\min\parbr{\frac{d_{\mathrm{E}}(u,v)}{d_{\mathcal{G}}(u,v)}\mid u,v\in V,\ u\neq v},\]where $d_{\mathcal{G}}$ is the length of the shortest path in $\mathcal{G}$.
\end{enumerate}
% There are some relationships among these parameters as follows:
% \begin{eqnarray}\label{paramsrelations}
%   s\leq r,\ s\leq 2\gamma,\ \rho\leq r,\ \delta\leq\paren{\frac{3r}{s}}^d.
% \end{eqnarray}
% Inequalities $s\leq r$ and $s\leq 2\gamma$ can be easily checked by definition. The third inequality comes from
% \[\rho\leq \frac{d_{\mathrm{E}}(u,v)}{d_{\mathcal{G}}(u,v)}\leq\frac{d_{\mathrm{E}}(u,v)}{1}\leq r.\]
% The last inequality is obtained by comparing volumes of spheres. For a node $u$, let $B\paren{u,r+\frac{s}{2}}$ be a $d$-dimensional ball centered in $u$ and with radius $r+\frac{s}{2}$. This ball includes a ball $B\paren{v,\frac{s}{2}}$ for all $v$ which is a neighbor node of $u$. By definition of $s$, balls $B\paren{v,\frac{s}{2}}$ are not crossing with each other. Therefore, the volume of $B\paren{u,r+\frac{s}{2}}$ is larger than the sum of the volumes of $B\paren{v,\frac{s}{2}}$ for all $v$. Then we obtain
% \[\paren{r+\frac{s}{2}}^d\geq \delta\paren{\frac{s}{2}}^d,\]
% and, by $s\leq r$,
% \[(3r)^d\geq \delta s^d,\]
% which leads the last inequality of \eqref{paramsrelations}.

The previous study has shown an asymptotic behavior of $J(P)$ with respect to $n$ when $P$ is on geometric graphs and the parameters above are fixed.
\begin{proposition}[Theorem 4.1 in \cite{mainp}]\label{geo}
  Let $P\in\mathbb{R}^{n\times n}$ be a reversible consensus matrix with invariant measure $\bm{\pi}$, associated with a graph $\mathcal{G}=(V,\mathcal{E})$. Assume that all nonzero entries of $P$ belong to the interval $[p_{\min},p_{\max}]$and that $\mathcal{G}$ is a geometric graph with parameters $(s,r,\gamma,\rho)$ and nodes lying in $\mathcal{Q}=[0,l]^d$ in which $\gamma<l/4$. Then,
  \[k_1+q_1f_\mathrm{d}(n)\leq J(P)\leq k_2+q_2f_\mathrm{d}(n),\]
  where
  \[f_\mathrm{d}(n)=\begin{cases}
    n & \textrm{if }d=1,\\
    \log n & \textrm{if }d=2,\\
    1 & \textrm{if }d\geq 3,\\
  \end{cases}\]
  and where $k_1,k_2,q_1$, and $q_2$ are positive numbers which are functions of the following parameters only: the dimension $d$, the geometric parameters of the graph ($s,r,\gamma$,and $\rho$), the maximum degree $\delta$, $p_{\min}$,and $p_{\max}$, and the products $\pi_{\min} n$ and $\pi_{\max} n$.
\end{proposition}

\Cref{geo} is useful when we consider a growing family of geometric graphs whose parameters $s,r,\gamma,\rho$ are bounded. However, \Cref{geo} is not for nonreversible consensus matrices.

Using \Cref{mainres2}, \Cref{geo} can be generalized as follows.

\begin{theorem}\label{maingeo}
  Let $P\in\mathbb{R}^{n\times n}$ be a consensus matrix with invariant measure $\bm{\pi}$, and let $\mathcal{G}(P)$ be the \correction{associated undirected} graph of $P$. Assume that all the nonzero entries of $P$ belong to the interval $[p_{\min},p_{\max}]$and that $\mathcal{G}(P)$ is a geometric graph with parameters $(s,r,\gamma,\rho)$ and nodes lying in $\mathcal{Q}=[0,l]^d$ in which $\gamma<l/4$. Then,
  \begin{eqnarray*}\label{geo1}
    J(P)\leq k'_2+q'_2f_\mathrm{d}(n),
  \end{eqnarray*}
  and particularly, if $P^*P=PP^*$ holds, then
  \begin{eqnarray*}\label{geo2}
    J(P)\geq k'_1+q'_1f_\mathrm{d}(n),
  \end{eqnarray*}
  where
  \[f_\mathrm{d}(n)=\begin{cases}
    n & \textrm{if }d=1,\\
    \log n & \textrm{if }d=2,\\
    1 & \textrm{if }d\geq 3,\\
  \end{cases}\]
  and where $k'_1,k'_2,q'_1$, and $q'_2$ are positive numbers which are functions of the following parameters only: the dimension $d$, the geometric parameters of the graph $(s,r,\gamma,\rho)$, the maximum degree $\delta$ of $\mathcal{G}(P)$, $p_{\min}$ and $p_{\max}$, and the products $\pi_{\min} n$ and $\pi_{\max} n$.
\end{theorem}

% \begin{proof}
%   Theorem 4.1 in \cite{mainp} has shown that for reversible $P$,
%   \begin{eqnarray}
%     \label{jpgeo}c_l\averes(\mathcal{G})\leq J(P)\leq c_u\averes(\mathcal{G})
%   \end{eqnarray}
%   with $c_l$ and $c_u$ dependent on $p_{\min},p_{\max},\delta,\pi_{\max}n,\pi_{\min}n$. Using \Cref{mainres2} in this paper, \eqref{jpgeo} can be generalized to
%   \[J(P)\leq c'_u\averes(\mathcal{G}(P)),\]
%   without the assumption of reversibility of $P$, where $c'_u$ depends on $p_{\min}$,$p_{\max}$,$\delta$,$\pi_{\max}n$, and $\pi_{\min}n$. Moreover, if $P^*P=PP^*$, then
%   \[J(P)\geq c'_l\averes(\mathcal{G}(P)),\]
%   with $c'_l$ dependent on $p_{\min},p_{\max},\delta,\pi_{\max}n,\pi_{\min}n$. Notice that the lower bound of \Cref{mainres2} uses $\din$, but because $\din\leq \delta$, we can define $c'_l$ as a variable dependent on $\delta$ instead of $\delta_{\mathrm{in}}$.

%   The rest of the proof is totally analogous to Section 6 in \cite{mainp}.
% \end{proof}

We omit the proof because it is analogous to Section 6 in \cite{mainp} except the bounds of $J(P)$ should be evaluated by \Cref{mainres2}.
\color{black}

\Cref{maingeo} is useful for a growing family of geometric graphs $\mathcal{G}_n=(V_n,\mathcal{E}_n)$ where $V_n\subseteq [0,l_n]^d$ and with the geometric parameters $(s_n,r_n,\gamma_n,\rho_n)$ which are bounded as
\begin{eqnarray}\label{boundedparam}
  s_n\geq s,\ r_n\leq r,\ \gamma_n\leq\gamma,\ \rho_n\geq\rho.
\end{eqnarray}

This family of graphs is called \textit{a family of geometric graphs with bounded parameters}, including grid graphs with boundaries where the edges are connected only to the nearest neighbors \cite{mainp}. In this sense, this family can be understood as a generalization of the grid graph case.

Although \Cref{geo} discuss the case of fixed geometric parameters, the similar result to \Cref{geo} is also known for this family of geometric graphs with bounded parameters in \cite{mainp}. Here, we state the generalized version of it as a corollary.
\color{black}
\begin{corollary}\label{maingeo2}
  Let $P_n\in\mathbb{R}^{n\times n}$ be a family of consensus matrices with invariant measure $\bm{\pi}_n$, and let $\mathcal{G}(P_n)$ be the \correction{associated undirected} graph of $P_n$. Assume that all nonzero entries of $P_n$ belong to the interval $[p_{\min},p_{\max}]$ and all entries of $n\bm{\pi}$ belong to the interval $[\bar{\pi}_{\min},\bar{\pi}_{\max}]$ and that the family of $\mathcal{G}(P_n)$ is a family of geometric graphs with bounded parameters $(s_n,r_n,\gamma_n,\rho_n)$ and nodes lying in $\mathcal{Q}=[0,l_n]^d$ in which $\gamma_n<l_n/4$. Then,
  \begin{eqnarray*}\label{geo1asymp}
    J(P)\leq k''_2+q''_2f_\mathrm{d}(n),
  \end{eqnarray*}
  and particularly, if $P^*P=PP^*$ holds, then
  \begin{eqnarray*}\label{geo2asymp}
    J(P)\geq k''_1+q''_1f_\mathrm{d}(n),
  \end{eqnarray*}
  where $f_\mathrm{d}(n)$ is the same as in \Cref{maingeo}, and $k''_1,k''_2,q''_1$, and $q''_2$ are positive numbers which are functions of the following parameters only: the dimension $d$, the bound of geometric parameters $(s,r,\gamma,\rho)$ which satisfy \eqref{boundedparam}, $p_{\min},p_{\max},\bar{\pi}_{\min}$ and $\bar{\pi}_{\max}$.
\end{corollary}

\correction{
The proof can be done by combining the discussion of \Cref{maingeo} and Section 4.3 in \cite{mainp}.
}

\begin{remark}\label{growthln}
  For a family of geometric graphs with bounded parameters, $l_n$ must grow linearly to $n^{1/d}$ in $n\to\infty$ (Lemma 6.3 in \cite{mainp}). This feature is used for setting of $l_n$ in the following numerical experiment.
\end{remark}

Based on the above discussion, we check the behavior of $J(P)$ on a family of geometric graphs with bounded parameters by numerical experiment. The dimension $d$ is set to $2$ or $3$. The experiment is conducted in the following way:
\begin{itemize}
  \item \textbf{Construct a geometric graph}
  \begin{itemize}
    \item We fix the dimension $d$, the geometric parameters $(s,r,\gamma,\rho)$, an edge making probability $p_{\mathrm{e}}$, an edge direction probability $p_{\mathrm{d}}$, and the number of nodes $n$.
    \item We fix a constant $c$ and set $l_n:=cn^{1/d}$.
    \item We choose $n$ nodes in $[0,l_n]^d$ in the following way:
    \begin{enumerate}
      \item We choose the first node in $[0,l_n]^d$ according to a uniform distribution.
      \item We repeat selecting the next node in $[0,l_n]^d$ according to a uniform distribution independently to the previous nodes. If the Euclidean distance between the picked node and any previous node is less than $s$, then we discard this node. We stop repeating when the number of accepted nodes becomes $n$.
    \end{enumerate}
    \item We construct the edges in the following way: For all pairs of nodes $\{u,v\}$, if the distance between $u$ and $v$ is greater than $r$, we do not draw any edge between them. Otherwise, we draw an edge between $u$ and $v$ with probability $p_{\mathrm{e}}$ (independently to other edges).
    \item If the graph is not connected, we discard it.
    \item We calculate $\gamma_n$ and $\rho_n$, and if $\gamma_n>\gamma$ or $\rho_n<\rho$, we discard the graph. The method of judgment of $\gamma_n>\gamma$ and $\rho_n<\rho$ is in Appendix.
  \end{itemize}
  \item \textbf{Construct a consensus matrix}
  \begin{itemize}
    \item We fix the constants $b$, $\bar{\pi}_{\min}$, and $\bar{\pi}_{\max}$.
    \item We construct a matrix $P_n\in\mathbb{R}^{n\times n}$, where $P_{uv}=1$ if $u$ and $v$ are connected by an edge or $u=v$, and otherwise $P_{uv}=0$. This is a symmetric matrix and not a stochastic matrix.
    \item For all undirected edges $\{u,v\}$, we change $P_{uv}$ to $0$ with probability $p_{\mathrm{d}}$, or $P_{vu}$ to $0$ with probability $p_{\mathrm{d}}$, or do nothing to $P_{uv}$ and $P_{vu}$ with the probability $1-2p_{\mathrm{d}}$.
    \item If $P$ is not irreducible, we discard $P$ and return to ``construct a geometric graph'' step.
    \item We change all nonzero entries in $P_n$ to random values chosen by a uniform distribution on $[b,1]$. Then, we normalize each row of $P_n$ so that $P_n$ becomes a stochastic matrix. By this process, the elements of $P_n$ are guaranteed to be greater than or equal to $\frac{b}{b+\delta}$.
    \item We calculate the invariant measure $\bm{\pi}$ of $P$. If $n\pi_{\min}<\bar{\pi}_{\min}$ or $n\pi_{\max}>\bar{\pi}_{\min}$, we discard $P$ and return to ``construct a geometric graph'' step.
  \end{itemize}
  \item \textbf{Calculate $J(P)$}
  \begin{itemize}
    \item For $P$ constructed by ``construct a consensus matrix'' step, we calculate $J(P)$ by \eqref{lqcost}. The summation ends at $t=10^4$, or the first point where the change of sum is less than $10^{-5}$ for $10$ consective $t$.
  \end{itemize}
\end{itemize}

\begin{table}[t]
  \caption{$n$ used in the experiment}
  \label{table:expn}
  \centering
    \begin{tabular}{c|l}
    \hline
    $d$ & $n$ \\
    \hline
    2 & 25, 50, 75, 100, 125, 150, 175, 200, 225, 250, 275, 300 \\
    3 & 50, 150, 250, 350, 450, 550, 600, 650, 700, 750, 800 \\
    \hline
    \end{tabular}
\end{table}
\begin{table}[t]
  \caption{Parameters used in the experiment}
  \label{table:exprm}
  \centering
    \begin{tabular}{cccccccccc}
    \hline
    $s$ & $r$ & $\gamma$ & $\rho$ & $p_{\mathrm{e}}$ & $p_{\mathrm{d}}$ & $c$ & $b$ & $\bar{\pi}_{\min}$ & $\bar{\pi}_{\max}$ \\
    \hline \hline
    0.1 & 1.0 & 1.0 & 0.052 & 0.8 & 0.1 & 0.5 & 0.8 & 0.1 & 3.0\\
    \hline
    \end{tabular}
\end{table}

\begin{figure}[htbp]
  \begin{minipage}[htbp]{0.48\linewidth}
    \centering
    \includegraphics[keepaspectratio, scale=0.38]{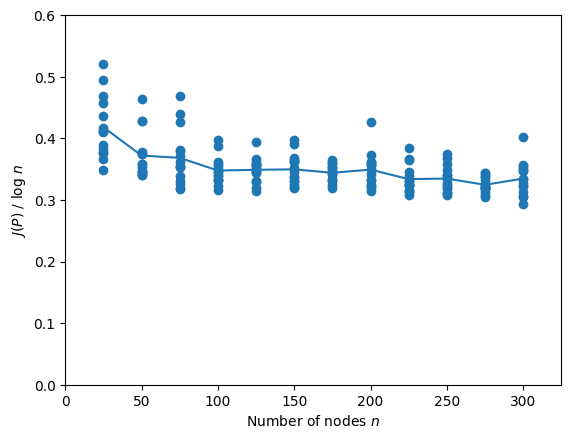}
  \end{minipage}
  \begin{minipage}[htbp]{0.48\linewidth}
    \centering
    \includegraphics[keepaspectratio, scale=0.38]{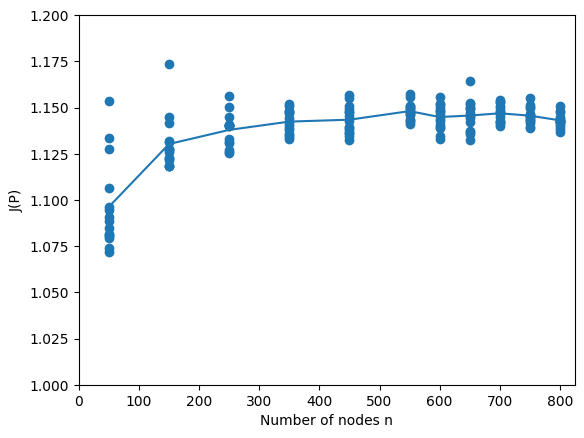}
  \end{minipage}
  \caption{The values of $J(P)/\log n$ in $d=2$ (left panel), and \correction{$J(P)$} in $d=3$ (right panel). The solid lines in both panels are the average of $J(P)$ (divided by $\log n$ in the left panel).}
  \label{figure:expres}
\end{figure}

We have run the construction and calculation for $d=2$ and $d=3$. In both cases, we have constructed $15$ consensus matrices for each $n$ listed in Table \ref{table:expn}\@. We also show the parameters used for the experiment in Table \ref{table:exprm} (common in $d=2$ and $d=3$).

Notice that we only specify $\gamma_n$ as equal or less than $\gamma$, so $\gamma_n$ can be larger than $l_n/4$ if $\gamma>l_n/4$. In fact, when $n$ is small, the constraint $\gamma_n<l_n/4$ can be violated. Particularly, in our parameter setting, the results in $n\leq 64$ of $d=2$ case and $n\leq 512$ of $d=3$ case are not guaranteed to obey \Cref{maingeo2}.

The result is shown in \Cref{figure:expres}. In the figure of $d=2$, the vertical axis shows $J(P)$ divided by $\log n$. The results suggest that the asymptotic growth of $J(P)$ is bounded as \Cref{maingeo2}. In fact, in small $n$, it seems that there is a decreasing feature in $d=2$ and an increasing feature in $d=3$, but in larger $n$ which satisfies $\gamma<l_n/4$, the asymptotic behavior predicted in \Cref{maingeo2} can be read. In addition, the lower bound for $J(P)$ is not guaranteed since our experiment does not assume $P^*P=PP^*$, but in the results, the value of $J(P)$ appears not to become much smaller. There may be some lower bound that holds for a random geometric graph with high probability.

%%%%%%%%%%%%%%%%%%%%%%%%%%%%%%%%%%%%%%%%%%%%%%%%%%%%%%%%%%%%%%%%%%%%%%%%%
\section{Conclusion}
\label{sec:conclusion}

In this paper, we presented an estimation of the LQ cost for the linear consensus algorithm applied to nonreversible matrices. Our approach leverages the reversible matrix $P^*P$, which remains reversible even when $P$ is not, to derive performance bounds using effective resistance. We further introduced novel concepts---the \textit{back-and-forth path} and \textit{pivot node}---to establish a relationship between the effective resistances of $\mathcal{G}(P)$ and $\mathcal{G}(P^*P)$. An application to geometric graphs was also demonstrated, underscoring the practical relevance of our results.

We believe that the methodology based on $P^*P$ can be extended to existing research on reversible cases. In reversible settings, since $P^*P = P^2$, studies that employ $P^2$ can be adapted to nonreversible cases using $P^*P$ instead. Future work may focus on establishing a more refined lower bound for the LQ cost in nonreversible cases, under assumptions that are less restrictive than those used in our current analysis, by leveraging effective resistance or alternative analytical tools.

\bibliographystyle{siamplain}
\bibliography{all}

%%%%%%%%%%%%%%%%%%%%%%%%%%%%%%%%%%%%%%%%%%%%%%%%%%%%%%%%%%%%%%%%%%%%%%%%%
%%%%%%%%%%%%%%%%%%%%%%%%%%%%%%%%%%%%%%%%%%%%%%%%%%%%%%%%%%%%%%%%%%%%%%%%%
\appendix

\section{Calculation of $\gamma_n$ and $\rho_n$ in the numerical experiment}
In the numerical experiment in Section \ref{subsection_geometric}, we have to judge whether $\gamma_n>\gamma$, and whether $\rho_n<\rho$. This condition is checked in the following way.
\begin{itemize}
  \item $\gamma_n>\gamma$ \\
  We take all points in $\mathcal{Q}=[0,l]^d$ so that each coordinate is an integer multiple of $\tilde{l}:=l/30$. Notice that the number of points is $31^d$. For all of these $31^d$ points, we check if they are more than $\gamma-\tilde{l}\sqrt{d}/2$ away from all the points in $V$. If not, it is guaranteed that all the points in $\mathcal{Q}$ are within distance $\gamma$ of any point in $V$, so we judge that $\gamma_n\leq\gamma$. Otherwise, there may be a point further than $\gamma$ from all points in $V$, so we judge that $\gamma_n>\gamma$ and discard the constructed graph. It is possible to be discarded even if $\gamma_n\leq\gamma$, but if the graph is not discarded, then $\gamma_n\leq\gamma$ necessarily holds.
  \item $\rho_n<\rho$ \\
  We use Floyd-Warshall (see Section 23.2 in \cite{cormen2022introduction}) method to calculate the distance between all pairs of points. Then, we calculate $\rho_n$ and judge whether $\rho_n<\rho$.
\end{itemize}

\end{document}